\documentclass[preprint,review,10pt]{elsarticle}
\usepackage{amssymb}
\usepackage{amsfonts}
\usepackage{amsmath}
\usepackage{amsthm}
\usepackage{color}
\usepackage{graphicx}
\usepackage{epsfig,mathrsfs}
\usepackage[bf,SL,BF]{subfigure}
\usepackage{fancyhdr}
\usepackage{CJK}
\usepackage{caption}
\usepackage{wrapfig}
\usepackage{cases}
\usepackage{subfigure}
\graphicspath{{figure/}}

\newtheorem{theorem}{Theorem}[section]
\newtheorem{lemma}{Lemma}[section]
\newtheorem*{lemma*}{Lemma A}

\newtheorem{proposition}{Proposition}[section]

\numberwithin{equation}{section}

\renewcommand{\eqref}[1]{(\ref{#1})}

 \newtheorem{remark}{Remark}[section]
 
 \allowdisplaybreaks

\begin{document}
 \pagenumbering{arabic}

\begin{frontmatter}
\title{{\bf \noindent Analysis of  (shifted) piecewise quadratic polynomial  collocation for nonlocal diffusion model}}
\author{Minghua Chen$^1$, Jiankang Shi$^1$ and  Xiaobo Yin$^2$}
\cortext[cor2]{Corresponding author. E-mail: chenmh@lzu.edu.cn.}
\address{$^1$School of Mathematics and Statistics,
Lanzhou University, Lanzhou 730000, P. R. China

$^2$Hubei Key Laboratory of Mathematical Sciences, School of Mathematics and Statistics, Central China Normal University, Wuhan 430079, China }
\date{}

\begin{abstract}
The piecewise quadratic polynomial collocation is used to approximate the nonlocal model, which generally  obtain  the {\em nonsymmetric indefinite system} [Chen et al., IMA J. Numer. Anal., (2021)].
In this case, the discrete maximum principle is not satisfied, which might be trickier for the stability analysis of the high-order numerical schemes [D'Elia et al., Acta Numer., (2020);
Leng et al., SIAM J. Numer. Anal., (2021)].
Here, we present the modified (shifted-symmetric) piecewise quadratic polynomial collocation for solving the linear nonlocal diffusion model, which has the {\em symmetric positive definite system} and satisfies the discrete maximum principle.
Using Faulhaber's formula and Riemann zeta function, the perturbation error for symmetric positive definite system  and nonsymmetric indefinite  systems are given.
Then the detailed proof of the convergence analysis for the nonlocal models with the general horizon parameter $\delta=\mathcal{O}\left(h^\beta\right)$, $\beta\geq0$ are provided.
More concretely, the global error is $\mathcal{O}\left(h^{\min\left\{2,1+\beta\right\}}\right)$ if $\delta$ is not set as a grid point, but it shall recover $\mathcal{O}\left(h^{\max\left\{2,4-2\beta\right\}}\right)$ when $\delta$ is set as a grid point.
We also prove that the shifted-symmetric scheme is asymptotically compatible, which has the global error $\mathcal{O}\left(h^{\min\left\{2,2\beta\right\}}\right)$ as $\delta,h\rightarrow 0$.
The numerical experiments (including two-dimensional case) are performed to verify the convergence.

\end{abstract}
\end{frontmatter}

\section{Introduction}\label{sec:into}
Nonlocal diffusion problems have been used to model very different scientific phenomena, which can either complement or serve as an alternative to classical  partial differential equations (PDEs).
The integral formulations of spatial interactions in nonlocal models with nonlocal Dirichlet volume constraint can naturally account for nonlocal interactions effects and allow more singular solutions  \cite{Du:2019}.
For example, nonlocal peridynamic (PD) is becoming an attractive emerging tool for the multiscale material simulations of crack nucleation and growth, fracture and failure of composites \cite{Silling:175--209}.
Mathematical analysis of PD models and other related nonlocal models, such as nonlocal diffusion, can be found in \cite{Du:2019}.
In particular, nonlocal models \cite{Gunzburger:676--696} with a finite range $\delta$ of interaction serve as a bridge between fractional PDEs \cite{Chen:1418--1438,Defterli:342--360,DDGGTZ:20} with $\delta\rightarrow\infty$ and local PDEs with $\delta\rightarrow0$.
For $\delta>0$, compared with classical PDE models, the complexities are introduced by the nonlocal interactions.
As $\delta\rightarrow0$, the nonlocal effect vanishes and the zero-horizon limit of nonlocal PD models reduce to a classical local PDEs model when the latter is well-defined.
Such limiting behavior provides connections and consistencies between nonlocal and local models, and has immense practical significance especially for multiscale modeling and simulations.
If $\delta$ to be proportional to $h$, as $\delta, h\rightarrow0$, the concept of asymptotically compatible  schemes have been introduced in \cite{Tian:1641--1665,TD:20,Zhang:52--68}.

There has been much recent interest in developing numerical algorithms for nonlocal models, including finite difference \cite{Tian:46--67,Tian:3458--3482}, finite element \cite{DuYIN:1913--1935,Tian:3458--3482,Tian:1641--1665,Zhang:213--236}, collocation method \cite{Tian:815--825,Wang:114--126,Zhang:52--68}, meshfree method \cite{Lehoucq:366--380}, multigrid method \cite{Chen:869--890} and conjugate gradient method \cite{Wang:114--126}.
Among various techniques for solving integral nonlocal problems, collocation methods are the simplest, since they only need a single integration and are much simpler to implement on a computer.
It is well-known that the standard piecewise linear polynomial collocation is used to approximate the nonlocal problems, which is not asymptotically compatible \cite{Tian:1641--1665,Zhang:52--68}.
To overcome it, the quadrature-based  piecewise linear polynomial collocation-type (finite difference schemes) are introduced \cite{Tian:1641--1665} and extended to the multidimensional case \cite{DTTY:607--625}.
Recently, a new corrected quadrature rule based on piecewise linear polynomial collocation is derived, which is also asymptotically compatible scheme \cite{Zhang:52--68}.

To seek numerical discretization of the strong form (e.g., collocation, finite difference), it is difficult to show stability of the high-order numerical schemes while trying to keep the asymptotically compatible property \cite{DDGGTZ:20,LTTF:21}, since the discrete maximum principle is not satisfied.
How about piecewise quadratic polynomial (high-order) collocation with {\em standard} quadrature rule?
We know that the piecewise quadratic polynomial collocation is used to approximate the nonlocal model, which derives the {\em nonsymmetric indefinite system} \cite{CCNWL:20,Chen:1--30}; namely, the discrete maximum principle is not satisfied.
How about the horizon parameter $\delta$ is not set as the grid point?
Since the whole key point of asymptotically compatible scheme is that convergent numerical methods which does not require any relation between mesh size/grid points and $\delta$ \cite{Du:2019}.
Moreover, the horizon of the material is a physical property of the material of the finite bar and should be independent of the computational mesh size $h$ \cite{WT:7730--7738}.
In fact, it is difficult to require that $\delta$ is set as the grid point in the multidimensional space.
The core object of this paper is to fill in this gap.
The main contribution of this work as follows.
We first present the shifted-symmetric  piecewise quadratic polynomial collocation  for nonlocal problem, which has the {\em symmetric positive definite system} and satisfies the discrete maximum principle.
Then the detailed proof of the convergence analysis for the nonlocal models with the general  horizon parameter $\delta=\mathcal{O}\left(h^\beta\right)$, $\beta\geq0$ are provided.
Namely, the global truncation  error is  $\min\left\{2,1+\beta\right\}$-order convergence if $\delta$ is not set as a grid point.
However, it shall recover $\max\left\{2,4-2\beta\right\}$-order convergence when $\delta$ is set as a grid point.
For asymptotically compatible scheme,  the global error is  $\min\left\{2,2\beta\right\}$-order convergence as $\delta,h\rightarrow 0$.
Finally, the two-dimensional numerical experiments  are performed to verify the convergence rate.

The paper is organized as follows. In the next section, we provide the standard collocation  and shifted-symmetric collocation
for the nonlocal problems by the piecewise quadratic polynomial, respectively.
In Section $\S3$, we analyze the local truncation errors with the general horizon parameter.
The global convergence rates for nonlocal models are detailed proved in Section $\S4$.
To show the effectiveness of the presented schemes, results of numerical experiments (including two-dimensional case) are reported in Section $\S5$.

\section{Collocation Method and Numerical Schemes}\label{sec:manu}

Let us consider the nonlocal model with a volumetric constraint domain \cite{Gunzburger:676--696}
\begin{equation}\label{2.1}
\left\{\begin{split}
- \mathcal{L}_\delta u_\delta(x)      &=f_{\delta}(x)   &  ~~{\rm in}  & ~~\Omega,\\
                u_\delta(x)           &=g_{\delta}(x)   &  ~~{\rm on}  & ~~\Omega_\mathcal{I}.
\end{split}\right.
\end{equation}
For the sake of simplicity, we denote $u_\delta(x)$ as $u(x)$ . Here
\begin{equation}\label{2.2}
\mathcal{L}_\delta u(x)=\int_{B_\delta(x)}[u(y)-u(x)]\gamma_\delta(x,y)dy,
~~\forall x \in \Omega
\end{equation}
and  $B_\delta(x)=\{y \in \mathbb{R}: |y-x|<\delta \}$ denotes a neighborhood centered at $x$ of radius $\delta$.
To keep the expression simple below we assume we are on the unit interval $\Omega$ with the volumetric constraint domain $\Omega_\mathcal{I}=[-\delta,0]\cup[1,1+\delta]$ for  one-dimensional model \eqref{2.1}. The specific form of such nonlocal interactions is prescribed by a nonnegative and radial kernel function $\gamma_\delta=\gamma_\delta(|y-x|) $.
We notice that there are many different choices to prescribe $\gamma_{\delta}(z)$ for nonlocal problem (\ref{2.1}),
such as a constant interaction kernel, weakly singular kernel et al., see the literatures \cite{DuYIN:1913--1935,Tian:3458--3482,Zhang:52--68}.
For illustration, we use the constant interaction kernel, a choice often used in many existing studies \cite{Chen:869--890,Du:2017,SP:11,Tian:3458--3482}, namely,
\begin{equation}\label{2.3}
\gamma_{{\delta}}(z)=\frac{3}{{\delta}^{3}},
\end{equation}
which has the finite second order moment of $\gamma_\delta$, i.e.,
\begin{equation}\label{2.4}
 C_\delta=\int_{0}^\delta z^2\gamma_\delta(z)dz=1,
\end{equation}
where $C_\delta$ is a well-defined elastic modulus or diffusion constant.

 Then, we can rewrite $\eqref{2.2}$ as
\begin{equation}\label{2.6}
\mathcal{L}_\delta u(x)  =\frac{3}{\delta^{3}}\int_{x-\delta}^{x+\delta} u(y)-u(x)dy =\frac{3}{\delta^{3}}\int_{0}^{\delta} u(x+z)-2u(x)+u(x-z)dz.
\end{equation}

Noted that if $u(x)$ is smooth enough or say $u\in C_c^\infty$, there exists \cite[p.\,92]{Du:2017}
\begin{equation}\label{2.5}
\mathcal{L}_\delta u(x)= C_\delta u''(x)  + \mathcal{O} \left(\delta^ 2\right).
\end{equation}

\subsection{Standard collocation method for integral operator \eqref{2.2}}
Now, we introduce and discuss the discretization scheme of (\ref{2.1}). Denote the horizon $\delta$ and the mesh size $h$ as
\begin{equation}\label{2.7}
\delta=\left\{
\begin{split}
&\widetilde{\delta}+r_0 h , ~~~ \delta\geq h,\\
&\widetilde{\delta}-r_0 h ,~ ~~ \delta\leq h
\end{split}
\right.
 ~~{\rm with}~~\widetilde{\delta}=rh,~~0\leq r_0 <1.
\end{equation}
Here the ratio $r=\big\lfloor\delta/h\big\rfloor\ge1$ if $\delta\geq h$ and $r=\big\lceil\delta/h\big\rceil=1$
if $\delta\le h$.
Let the mesh points $x_i=ih$, $h=1/(N+1)$, $i \in \mathcal{N}$ and
\begin{equation*}
\begin{split}
&\mathcal{N}     =\left\{-r,{-r+\frac{1}{2}},\ldots,0,\frac{1}{2},\ldots,{N+\frac{1}{2}},{N+1},\ldots,{N+r+\frac{1}{2}},N+r+1\right\},\\
&\mathcal{N}_{in}=\left\{{\frac{1}{2}},1,\frac{3}{2}, \ldots, N-\frac{1}{2}, N,{N+\frac{1}{2}}\right\},~~\mathcal{N}_{out}=\mathcal{N}\setminus\mathcal{N}_{in}.
\end{split}
\end{equation*}
\begin{remark}\label{remarkad2.1}
We choose the exact solution of the nonlocal diffusion problem (\ref{2.1}) as $u_\delta(x)=x^2(1-x^2)$. This naturally leads to a $\delta$-dependent right-hand side $f_{\delta}(x)=12x^2-2+\frac{6}{5}\delta^2$,
which yields $f_{\delta}(x)-f_{\widetilde{\delta}}(x)=\mathcal{O}\left(h\right)$ if a constant $\delta\neq \widetilde{\delta}$.
We will see that the global error is $\mathcal{O}\left(h^4\right)$ if $\delta$ is set as the grid point, but it shall drop down to $\mathcal{O}\left(h\right)$ if $\delta$ is not set as the grid point.
It implies that we shall need to carefully consider the case where $\delta$ is not set as the grid point in \eqref{2.7}, i.e., $r_0 \neq 0$.
\end{remark}

Let the piecewise quadratic basis functions $\phi_j(x)$ and $\phi_{j+\frac{1}{2}}(x)$ be given in \cite[p.\,37]{Aikinson:2009} and $u_{Q}(x)$ be the piecewise Lagrange quadratic interpolant of $u(x)$, i.e.,
\begin{equation}\label{2.8}
u_{Q}(x) =\sum_{j=-r}^{N+r+1}u(x_j)\phi_j(x)+\sum_{j=-r}^{N+r}u(x_{j+\frac{1}{2}})\phi_{j+\frac{1}{2}}(x).
\end{equation}

Noted that the piecewise quadratic polynomial \eqref{2.8} is used to approximate the nonlocal operator $\mathcal{L}_{\widetilde{\delta}}u(x)$ rather than $\mathcal{L}_\delta u(x)$ in \eqref{2.6}, which plays a key role in the design of numerical scheme for nonlocal models \eqref{2.1} when $\delta$ is not set as the grid point, see Remark \eqref{remarkad2.1}.
Let us first consider the following piecewise quadratic interpolation of $\mathcal{L}_{\widetilde{\delta}} u(x)$ (and $\mathcal{L}_{{\delta}} u(x)$ later on).
\begin{equation}\label{2.9}
\mathcal{L}_{\widetilde{\delta},h} u(x_\frac{i}{2})
=\frac{3}{\widetilde{\delta}^{3}}\int_{x_{\frac{i}{2}-r}}^{x_{\frac{i}{2}+r}}\sum_{j=-r}^{N+r+1}u(x_j)\phi_j(y) + \sum_{j=-r}^{N+r}u (x_{j+\frac{1}{2}} )\phi_{j+\frac{1}{2}}(y) - u (x_{\frac{i}{2}} )dy
\end{equation}
with $1\leq i\leq 2N+1 $. Taking $\eta^{h}_{\widetilde{\delta}}=\frac{h}{2\widetilde{\delta}^{3}}$, we can divide $\eqref{2.9}$ into two parts as follows:
\begin{equation}\label{ada2.11}
\begin{split}
\mathcal{L}_{\widetilde{\delta},h} u(x_i)
&=\eta^{h}_{\widetilde{\delta}}\left[\sum_{j=i-r}^{i+r}a_{|i-j|}u(x_j)+\sum_{j=i-r}^{i+r-1}a_{|i-j-\frac{1}{2}|}u(x_{j+\frac{1}{2}})\right]\\
&=\eta^{h}_{\widetilde{\delta}} \left[\sum_{m=-r}^{r}{a_m}u(x_{i+m})+\sum_{m=-r}^{r-1}{a_{m+\frac{1}{2}}}u(x_{i+m+\frac{1}{2}})\right],~~1\leq i\leq N
\end{split}
\end{equation}
with
\begin{equation}\label{aadd2.10}
\begin{split}
 & a_{0}=2-12r, \quad a_{m}=2,~~1\leq m\leq r-1, \quad a_{r}=1,\\
 & a_{m+\frac{1}{2}}=4, ~~0\leq m\leq r-1.
\end{split}
\end{equation}
On the other hand, we have
\begin{equation}\label{addd2.0012}
\begin{split}
\mathcal{L}_{\widetilde{\delta},h} u(x_{i+\frac{1}{2}})
&=\eta^{h}_{\widetilde{\delta}}\left[\sum_{j=i-r}^{i+r+1}c_{|i+\frac{1}{2}-j|-\frac{1}{2}}u(x_j)+\sum_{j=i-r}^{i+r}d_{|i-j|}u(x_{j+\frac{1}{2}})\right]\\
&=\eta^{h}_{\widetilde{\delta}}\left[\sum_{m=-r}^{r+1}c_{|\frac{1}{2}-m|-\frac{1}{2}}u(x_{i+m})+\sum_{m=-r}^{r}d_{|m|}u(x_{i+m+\frac{1}{2}})\right],~~0\leq i\leq N
\end{split}
\end{equation}
with
\begin{equation}\label{2.10}
\begin{split}
&c_{m}=2, ~~0\leq m\leq r-2, \quad c_{r-1}=\frac{9}{4}, \quad c_{r}=-\frac{1}{4},\\
&d_{0}=4-12r, \quad d_{m}=4, ~~1\leq m\leq r-1, \quad d_{r}=2.\\
\end{split}
\end{equation}
\begin{remark}
From \eqref{aadd2.10} and \eqref{2.10}, the standard piecewise quadratic polynomial collocation is used to approximate the nonlocal model \eqref{2.1}, which derives the nonsymmetric indefinite system.
In this case, the discrete maximum principle  is not satisfied, which might be trickier for the stability analysis of the high-order numerical schemes \cite{DDGGTZ:20,LTTF:21}.
This phenomenon also arises in \cite{CCNWL:20,Chen:1--30}. This motives us to construct the shifted-symmetric collocation method, which has the symmetric positive definite system  and satisfies the discrete maximum principle.
\end{remark}
\subsection{Shifted-symmetric collocation method for integral operator \eqref{2.2}}
We now construct the shifted-symmetric collocation method for nonlocal integral operator.
Namely, from \eqref{ada2.11}, we have
\begin{equation}\label{2.11}
\begin{split}
\mathcal{L}^{S}_{\widetilde{\delta},h} u(x_i)
&=\mathcal{L}_{\widetilde{\delta},h} u(x_i)\\
&=\eta^{h}_{\widetilde{\delta}} \left[\sum_{m=-r}^{r}{a_m}u(x_{i+m})+\sum_{m=-r}^{r-1}{a_{m+\frac{1}{2}}}u(x_{i+m+\frac{1}{2}})\right],~~1\leq i\leq N,\\
\mathcal{L}^{S}_{\widetilde{\delta},h} u(x_{i+\frac{1}{2}})
&:=\eta^{h}_{\widetilde{\delta}} \left[\sum_{m=-r}^{r}{a_m}u(x_{i+m+\frac{1}{2}})+\sum_{m=-r}^{r-1}{a_{m+\frac{1}{2}}}u(x_{i+m+1})\right],~~0\leq i\leq N
\end{split}
\end{equation}
with $\eta^{h}_{\widetilde{\delta}}=\frac{h}{2\widetilde{\delta}^{3}}.$
Here the coefficients are given in \eqref{aadd2.10}, i.e,
\begin{equation}\label{2.200000}
\begin{split}
 & a_{0}=2-12r, \quad a_{m}=2,~~1\leq m\leq r-1, \quad a_{r}=1,\\
 & a_{m+\frac{1}{2}}=4, ~~0\leq m\leq r-1.
\end{split}
\end{equation}
In particular, if $\delta\leq h$ (i.e., $\widetilde{\delta}=h$), the coefficients in \eqref{2.200000} reduce to
\begin{equation}\label{2.22222}
a_{0}=-10, \quad a_{1}=1, \quad a_{\frac{1}{2}}=4.
\end{equation}
We can easily get
\begin{equation}\label{2.16}
a_{0}+2\sum_{m=1}^{r}a_{m}+2\sum_{m=0}^{r-1}a_{m+\frac{1}{2}}=0.
\end{equation}

\subsection{Shifted-symmetric collocation method for nonlocal model \eqref{2.1}}
Let $u_{\delta,h}(x_{i})$ be the approximated value of $u_{\delta}(x_{i})$ and $f_{\delta,i} = f_{\delta}(x_{i})$.
From \eqref{2.11}, the shifted-symmetric collocation of \eqref{2.1} has the following symmetric systems
\begin{equation}\label{2.17}
\left\{\begin{split}
-\mathcal{L}^{S}_{\widetilde{\delta},h} u_{\delta,h}(x_{i}) &=f_{\delta,i}, ~~ i \in \mathcal{N}_{in},\\
                     u_{\delta,h}(x_{i})                    &=g_{\delta,i}, ~~ i \in \mathcal{N}_{out}.
\end{split}\right.
\end{equation}

For simplicity, we denote $u_{i}$ as $u_{\delta,h}(x_i)$ and $g_{i}$ as $g_{\delta,i}$ with the following sketch that characterizes different variables:
\begin{equation*}
\begin{split}
&\big[
\underset{\text{boundary points}}{\underbrace{  x_{-r+\frac{1}{2}} \ldots  x_{0},}}~~
\underset{\text{interface points}}{\underbrace{ x_{\frac{1}{2}} \ldots  x_{r},}}~
\underset{\text{internal points}}{\underbrace{  x_{r+\frac{1}{2}}\ldots  x_{N-r+\frac{1}{2}},}}~
\underset{\text{interface points}}{\underbrace{ x_{N-r+1}  \ldots  x_{N+\frac{1}{2}},}}~
\underset{\text{boundary points}}{\underbrace{  x_{N+1} \ldots  x_{N+r+\frac{1}{2}}}}
\big]
\end{split}
\end{equation*}
\begin{equation*}
\begin{split}
&\big[
\underset{\text{boundary values}}{\underbrace{  g_{-r+\frac{1}{2}}\ldots   g_{0},}}~~
\underset{\text{interface values}}{\underbrace{ u_{\frac{1}{2}}\ldots  u_{r}},}~
\underset{\text{internal values}}{\underbrace{  u_{r+\frac{1}{2}}\ldots u_{N-r+\frac{1}{2}},}}~
\underset{\text{interface values}}{\underbrace{ u_{N-r+1} \ldots u_{N+\frac{1}{2}},}}~
\underset{\text{boundary values}}{\underbrace{  g_{N+1} \ldots  g_{N+r+\frac{1}{2}}}}
\big].
\end{split}
\end{equation*}
Let
\begin{equation*}
\begin{split}
U_{\delta,h}&=[u_{1},u_{2},\ldots,u_{N},u_{\frac{1}{2}},u_{\frac{3}{2}},\ldots,u_{N+\frac{1}{2}}]^T;\\
F_{\delta,h}&=[f_{\delta,1},f_{\delta,2},\ldots,f_{\delta,N},f_{\delta,\frac{1}{2}},f_{\delta,\frac{3}{2}},\ldots,f_{\delta,N+\frac{1}{2}}]^T.
\end{split}
\end{equation*}
Let
\begin{equation*}
w_{1}=\left(a_{1},a_{2},\ldots,a_{r}\right)^T, ~~w_{2}=\left(a_{\frac{1}{2}},a_{\frac{3}{2}},\ldots,a_{r-\frac{1}{2}}\right)^T, ~~w_{3}=\left(a_{\frac{3}{2}},a_{\frac{5}{2}},\ldots,a_{r-\frac{1}{2}}\right)^T;
\end{equation*}
and
\begin{equation*}
\begin{split}
G_{1}&={\rm triu}\left[{\rm toeplitz}\left(g_{0}, g_{-1}, \ldots, g_{-r+1}\right)\right],\\
G_{2}&={\rm triu}\left[{\rm toeplitz}\left(g_{-\frac{1}{2}}, g_{-\frac{3}{2}}, \ldots, g_{-r+\frac{3}{2}}\right)\right],\\
G_{3}&={\rm triu}\left[{\rm toeplitz}\left(g_{N+1}, g_{N+2}, \ldots, g_{N+r}\right)\right],\\
G_{4}&={\rm triu}\left[{\rm toeplitz}\left(g_{N+\frac{3}{2}}, g_{N+\frac{5}{2}}, \ldots, g_{N+r-\frac{1}{2}}\right)\right],\\
G_{5}&={\rm triu}\left[{\rm toeplitz}\left(g_{-\frac{1}{2}}, g_{-\frac{3}{2}}, \ldots, g_{-r+\frac{1}{2}}\right)\right],\\
G_{6}&={\rm triu}\left[{\rm toeplitz}\left(g_{N+\frac{3}{2}}, g_{N+\frac{5}{2}}, \ldots, g_{N+r+\frac{1}{2}}\right)\right].
\end{split}
\end{equation*}
Similar to the discussion in \cite{Chen:869--890}, we introduce the following auxiliary vector
\begin{equation*}
\begin{split}
F_{\mathcal{V},L}=&\eta^h_{\widetilde{\delta}}\left[\left(G_{1}w_{1}+[G_{2}w_{3};0]\right); (0)_{(N-2r)\times1}; {\rm flip}\left(G_{3}w_{1}+[G_{4}w_{3};0]\right); (0)_{(N+1)\times1}\right],\\
F_{\mathcal{V},R}=&\eta^h_{\widetilde{\delta}}\left[(0)_{N\times1}; \left(G_{1}w_{2}+G_{5}w_{1}\right); (0)_{(N-2r)\times1}; {\rm flip}\left(G_{3}w_{2}+G_{6}w_{1}\right)\right].
\end{split}
\end{equation*}
Then the numerical scheme \eqref{2.17} can be recast as
\begin{equation}\label{2.18}
 A^{S}_{\widetilde{\delta},h}U_{\delta,h}=F^{S}_{\delta,h}~~{\rm with}~~ F^{S}_{\delta,h}=F_{\delta,h}+F_{\mathcal{V},L}+F_{\mathcal{V},R},
\end{equation}
where the stiffness matrix $A^{S}_{\widetilde{\delta},h}$ consists of the following four block-structured matrices with Toeplitz-like blocks,
\begin{equation}\label{2.19}
A^{S}_{\widetilde{\delta},h}=\eta^{h}_{\widetilde{\delta}}
\left[\begin{array}{cc}
\mathcal{A}& \mathcal{B} \\
\mathcal{B}^\mathsf{T}& \mathcal{\tilde{A}}
\end{array}\right]_{(2N+1)\times(2N+1)},
\end{equation}
and
\begin{equation*}
\begin{split}
&\mathcal{A}_{N\times N}=-{\rm toeplitz}(a_{0}, a_{1}, \ldots ,a_{r-1}, a_{r},(0)_{(N-r-1)\times1} );\\
&\mathcal{\tilde{A}}_{(N+1)\times (N+1)}=-{\rm toeplitz}(a_{0}, a_{1}, \ldots ,a_{r-1}, a_{r}, (0)_{(N-r)\times1} );\\
&\mathcal{B}_{N\times (N+1)}=-{\rm toeplitz}\left(\left[w_{2}; (0)_{(N-r)\times1}\right], \left[a_{\frac{1}{2}}; w_{2}; (0)_{(N-r)\times1}\right]\right),
\end{split}
\end{equation*}
where $w_{2}=\left(a_{\frac{1}{2}},a_{\frac{3}{2}},\ldots,a_{r-\frac{1}{2}}\right)^T$ and the coefficients $a_i$ are given in \eqref{2.200000}.

\begin{remark}[Nonsymmetric indefinite system]
From \eqref{ada2.11} and \eqref{addd2.0012}, the standard collocation method  of \eqref{2.1} has the following form
\begin{equation}\label{ad2.22}
\left\{\begin{split}
-\mathcal{L}_{\widetilde{\delta},h} u_{\delta,h}(x_{i}) &=f_{\delta,i}, ~~i \in \mathcal{N}_{in},\\
                  u_{\delta,h}(x_{i})                   &=g_{\delta,i},      ~~i \in \mathcal{N}_{out}.
\end{split}\right.
\end{equation}
Then the numerical scheme \eqref{ad2.22} can be recast as
\begin{equation}\label{adas2.22}
 A^{N}_{\widetilde{\delta},h}U_{\delta,h}=F^{N}_{\delta,h},
\end{equation}
where $F^{N}_{\delta,h}$ is similar computed in \eqref{2.19} and  the stiffness matrix $A^{N}_{\widetilde{\delta},h}$ is defined by
\begin{equation}
A^{N}_{\widetilde{\delta},h}=\eta^{h}_{\widetilde{\delta}}
\left[\begin{array}{cc}
\mathcal{A}& \mathcal{B} \\
\mathcal{C}& \mathcal{D}
\end{array}\right]_{(2N+1)\times(2N+1)}.
\end{equation}
Here
\begin{equation*}
\begin{split}
&\mathcal{D}_{(N+1)\times (N+1)}=-{\rm toeplitz}(d_{0}, d_{1}, \ldots ,d_{r-1}, d_{r}, (0)_{(N-r)\times1} );\\
&\mathcal{C}_{(N+1)\times N}=-{\rm toeplitz}\left( \left[c_0; w; (0)_{(N-r-1)\times1}\right], \left[w; (0)_{(N-r-1)\times1}\right]\right)\\
\end{split}
\end{equation*}
with $w=\left(c_0,c_1,\ldots,c_r\right)^T$.
\end{remark}

\begin{remark}[Symmetric systems for 2D]
The similar arguments can be performed as in  \eqref{2.18}, we obtain the   two-dimensional symmetric algebraic  systems
\begin{equation}\label{bd2.24}
 A^{S,2D}_{\widetilde{\delta},h}U^{2D}_{\delta,h}=F^{S,2D}_{\delta,h}~~{\rm with }~~ A^{S,2D}_{\widetilde{\delta},h}=C_{\widetilde{\delta},h} A^{S}_{\widetilde{\delta},h}  \otimes A^{S}_{\widetilde{\delta},h},
\end{equation}
where the coefficient $C_{\widetilde{\delta},h}$ is only dependent on $\widetilde{\delta}$ and $h$.
\end{remark}

\section{Local Truncation Error}
The techniques used to investigate the nonlocal models \eqref{2.1} include Taylor series expansions, Fourier transform, consistency of weak forms and nonlocal calculus of variations.
In this work, the technique is based on Taylor series expansions (which requires the suitable regularity assumption on the solution  \cite[pp.\,91-92]{Du:2017}) that has been among the most popular tools for nonlocal problems.
We next analyze the local truncation error  of the piecewise quadratic polynomial collocation for nonsymmetric indefinite systems \eqref{ad2.22} with the general horizon parameter.

\subsection{Local truncation error for nonsymmetric indefinite systems \eqref{ad2.22} with $\delta=\widetilde{\delta}$}
Let $u_C(x)$ be the piecewise Lagrange cubic interpolant of $u(x)$ in $x\in[x_{i-1},x_{i}]$, i.e.,
\begin{equation}\label{3.1}
\begin{split}
u_C(x)
=&-\frac{8(x-x_{i-3/4})(x-x_{i-1/2})(x-x_{i})}{h^3}u(x_{i-1})\\
&+\frac{64(x-x_{i-1})(x-x_{i-1/2})(x-x_{i})}{3h^3}u(x_{i-3/4})\\
&-\frac{16(x-x_{i-1})(x-x_{i-3/4})(x-x_{i})}{h^3}u(x_{i-1/2})\\
&+\frac{8(x-x_{i-1})(x-x_{i-3/4})(x-x_{i-1/2})}{3h^3}u(x_{i}).
\end{split}
\end{equation}

According to \eqref{2.8}, \eqref{3.1} and Taylor series expansion \cite[p.158]{Aikinson:2009}, there exists $\xi_{i-1}\in\left(x_{i-1},x_{i}\right)$ such that
\begin{equation}\label{3.2}
u_{C}(x)-u_Q(x)=\frac{u_C^{(3)}(\xi_{{i-1}}) }{3!}\left(x-x_{i-1}\right)\left(x-x_{i-\frac{1}{2}}\right)\left(x-x_{i}\right),~x\in[x_{i-1},x_{i}]
\end{equation}
with
\begin{equation}\label{b3.3}
\begin{split}
\frac{u_{C}^{(3)}(\xi_{{i-1}})}{3!}
&=\frac{8}{3h^3}\left[-3u(x_{i-1})+8u(x_{i-3/4})-6u(x_{i-1/2})+ u(x_{i}) \right]\\
&=\frac{u^{(3)}(\iota_{i})+3u^{(3)}(\zeta_{i})-u^{(3)}(\varsigma_{i})}{18},~\iota_{i},\zeta_{i},\varsigma_{i}\in(x_{i-1},x_{i}),
\end{split}
\end{equation}
where $\frac{u_{C}^{(3)}(\xi_{{i-1}})}{3!}$ is a constant, since $u_{C}(y)-u_{Q}(y)$ is a cubic polynomial.

\begin{lemma}\label{lemma3.3}
Let $\delta=\widetilde{\delta}=r h$. Let $u_{Q}(y)$ and $u_{C}(y)$ be defined by \eqref{2.8} and \eqref{3.1}, respectively. Then
\begin{equation*}
 Q_{\frac{i}{2}}:=\int_{x_{\lceil\frac{i}{2}\rceil-1}}^{x_{\lfloor\frac{i}{2}\rfloor+1}} u_{C}(y)-u_{Q}(y)dy=0
\end{equation*}
where $i$ is a positive integer number, $\lfloor\frac{i}{2}\rfloor$ and $\lceil\frac{i}{2}\rceil$ denote the greatest integer that is less than or equal to $\frac{i}{2}$ and the least integer that is greater than or equal to $\frac{i}{2}$, respectively.
\end{lemma}
\begin{proof}
If $i$ is even, we have
\begin{equation*}
\begin{split}
\int^{x_{\frac{i}{2}+1}}_{x_{\frac{i}{2}}}\left(y-x_{\frac{i}{2}}\right)\left(y-x_{\frac{i+1}{2}}\right)\left(y-x_{\frac{i}{2}+1}\right)dy = &0, \\
\int^{x_{\frac{i}{2}}}_{x_{\frac{i}{2}-1}}\left(y-x_{\frac{i}{2}-1}\right)\left(y-x_{\frac{i-1}{2}}\right)\left(y-x_{\frac{i}{2}}\right)dy = &0.
\end{split}
\end{equation*}
If $i$ is odd, it yields
\begin{equation*}
\begin{split}
\int_{x_{\frac{i}{2}}}^{x_{\frac{i+1}{2}}}\left(y-x_{\frac{i-1}{2}}\right)\left(y-x_{\frac{i}{2}}\right)\left(y-x_{\frac{i+1}{2}}\right)dy
& = h^{4}\int^{\frac{1}{2}}_{0}\left(t+\frac{1}{2}\right)t\left(t-\frac{1}{2}\right)dt,\\
\int^{x_{\frac{i}{2}}}_{x_{\frac{i-1}{2}}} \left(y-x_{\frac{i-1}{2}}\right)\left(y-x_{\frac{i}{2}}\right) \left(y-x_{\frac{i+1}{2}}\right)dy
& =- h^{4}\int^{\frac{1}{2}}_{0} \left(t+\frac{1}{2}\right)t\left(t-\frac{1}{2}\right)dt.
\end{split}
\end{equation*}
From \eqref{3.2}, there exist $\xi\in \left(x_{\frac{i-1}{2}},x_{\frac{i+1}{2}}\right)$ such that
$$u_{C}(y)-u_{Q}(y)=\frac{u_{C}^{(3)}\left(\xi\right)}{3!}\left(y-x_{\frac{i-1}{2}}\right)\left(y-x_{\frac{i}{2}}\right)\left(y-x_{\frac{i+1}{2}}\right)\quad\forall~y\in \left[x_{\frac{i-1}{2}},x_{\frac{i+1}{2}}\right].$$
Therefore, we have
\begin{equation*}
\begin{split}
Q_{\frac{i}{2}}
&=\int_{x_{\frac{i-1}{2}}}^{x_{\frac{i+1}{2}}}u_{C}(y)-u_{Q}(y)dy=\int_{x_{\frac{i-1}{2}}}^{x_{\frac{i}{2}}} u_{C}(y)-u_{Q}(y)dy + \int_{x_{\frac{i}{2}}}^{x_{\frac{i+1}{2}}}  u_{C}(y)-u_{Q}(y) dy \\
&=\frac{u_{C}^{(3)}\left(\xi\right) }{3!}h^{4} \left(\int^{\frac{1}{2}}_{0} \left(t+\frac{1}{2}\right)t\left(t-\frac{1}{2}\right)dt-\int^{\frac{1}{2}}_{0} \left(t+\frac{1}{2}\right)t\left(t-\frac{1}{2}\right) dt\right)=0.
\end{split}
\end{equation*}
The proof is completed.
\end{proof}
\begin{lemma}\label{lemma3.4}
Let $\delta=\widetilde{\delta}=r h$. Let $u_{Q}(y)$ and $u_{C}(y)$ be defined by \eqref{2.8} and \eqref{3.1}, respectively.
Then there exists $\xi_{\frac{i+1}{2}-r}\in\left(x_{\frac{i}{2}-r},x_{\frac{i+1}{2}-r}\right)$ such that
\begin{equation*}
Q_{l}: =\int^{x_{\lceil\frac{i}{2}\rceil-1}}_{x_{\frac{i}{2}-r}} u_{C}(y)-u_{Q}(y) dy
=\left\{ \begin{array}{l@{\quad}l}
                               0,                                     & {\rm if}~ i~ {\rm is~ even},\\
  - \frac{1}{384} u_{C}^{(3)}\left(\xi_{\frac{i+1}{2}-r}\right)h^{4}, & {\rm if}~ i~ {\rm is~ odd}.
\end{array}\right.
\end{equation*}
\end{lemma}
\begin{proof}
From \eqref{3.2}, there exists $\xi_{m}\in \left(x_{m-1},x_m\right)$ such that
$$u_{C}(y)-u_{Q}(y)=\frac{u_{C}^{(3)}(\xi_{m})}{3!}(y-x_{m-1})\left(y-x_{m-\frac{1}{2}}\right)(y-x_{m}), \quad\forall y\in[x_{m-1},x_m].$$
For the sake of simplicity, we take $w(\xi_{m})=\frac{u_{C}^{(3)}(\xi_{m})}{3!}$.
Then
\begin{equation*}
Q_l:=J_1+J_{2}
\end{equation*}
with
\begin{equation*}
\begin{split}
J_{1} =& \sum^{\lceil\frac{i}{2}\rceil-1}_{m={\lceil\frac{i}{2}\rceil-r+1}}w(\xi_{m}) \int^{x_{m}}_{x_{m-1}}  (y-x_{m-1})\left(y-x_{m-\frac{1}{2}}\right)(y-x_{m})dy;\\
J_{2} =& w\left(\xi_{\frac{i+1}{2}-r}\right)\int^{x_{\lceil\frac{i}{2}\rceil-r}}_{x_{\frac{i}{2}-r}} \left(y-x_{\frac{i-1}{2}-r}\right) \left(y-x_{\frac{i}{2}-r}\right)\left(y-x_{\frac{i+1}{2}-r}\right)dy.
\end{split}
\end{equation*}
Using $\int^{1}_{0}\tau\left(\tau-\frac{1}{2}\right)\left(\tau-1\right)d\tau=0$, it yields
\begin{equation}\label{3.3}
\int^{x_{m}}_{x_{m-1}} (y-x_{m-1})\left(y-x_{m-\frac{1}{2}}\right)(y-x_{m})dy = 0.
\end{equation}
It leads to $J_1 =0.$
If $i$ is even, we have $J_{2} =0.$

On the other hand, if $i$ is odd, we have
\begin{equation*}
\int_{x_{\frac{i}{2}-r}}^{x_{\frac{i+1}{2}-r}}\left(y-x_{\frac{i-1}{2}-r}\right)\left(y-x_{\frac{i}{2}-r}\right)\left(y-x_{\frac{i+1}{2}-r}\right)dy
=h^{4}\int^{\frac{1}{2}}_{0}\left(t+\frac{1}{2}\right)t\left(t-\frac{1}{2}\right)dt.
\end{equation*}
It leads to
\begin{equation}\label{ad3.a4}
J_{2} = w\left(\xi_{\frac{i+1}{2}-r}\right)h^{4}\int^{\frac{1}{2}}_{0}\left(t+\frac{1}{2}\right)t\left(t-\frac{1}{2}\right)dt= -\frac{1}{64} w\left(\xi_{\frac{i+1}{2}-r}\right)h^{4}.
\end{equation}
The proof is completed.
\end{proof}
\begin{lemma}\label{lemma3.5}
Let $\delta=\widetilde{\delta}=r h$. Let $u_Q(y)$ and $u_C(y)$ be defined by \eqref{2.8} and \eqref{3.1}, respectively.
Then there exists $\eta_{\frac{i-1}{2}+r}\in\left( x_{\frac{i-1}{2}+r},x_{\frac{i}{2}+r}\right)$ such that
\begin{equation*}
Q_{r}: =\int^{x_{\frac{i}{2}+r}}_{x_{\lfloor\frac{i}{2}\rfloor+1}} u_{C}(y)-u_{Q}(y) dy=
\left\{\begin{array}{l@{\quad} l}
                               0,                                  & {\rm if}~ i~ {\rm is~ even},\\
\frac{1}{384} u_{C}^{(3)}\left(\eta_{\frac{i-1}{2}+r}\right)h^{4}, & {\rm if}~ i~ {\rm is~ odd}.
\end{array}\right.
\end{equation*}
\end{lemma}
\begin{proof}
The similar arguments can be performed as Lemma \ref{lemma3.4}, we omit it here.
\end{proof}
\begin{lemma}\label{lemma3.6}
Let $\delta=\widetilde{\delta}=r h$. Let $Q_{l}$ and $Q_{r}$ be defined in Lemmas \ref{lemma3.4} and \ref{lemma3.5}, respectively. Then
\begin{equation*}
\left|Q_{l} +  Q_{r}\right|=\mathcal{O}\left(h^4\right)\delta.
\end{equation*}
\end{lemma}
\begin{proof}
From Lemmas \ref{lemma3.4} and \ref{lemma3.5} and \eqref{b3.3}, we get
\begin{equation*}
\begin{split}
\left| Q_{l} +  Q_{r}\right|
&= \left|{u_{C}^{\left(3\right)}\left(\xi_{\frac{i+1}{2}-r}\right)}-{u_{C}^{\left(3\right)}\left(\eta_{\frac{i-1}{2}+r}\right)}\right| \frac{h^4}{384} \\
&= \left|\frac{ u^{(3)}(\iota_{i})+3u^{(3)}(\zeta_{i})-u^{(3)}(\varsigma_{i})}{3}-\frac{ u^{(3)}(\widetilde{\iota}_{i})+3u^{(3)}(\widetilde{\zeta}_{i})-u^{(3)}(\widetilde{\varsigma}_{i})}{3}\right| \frac{h^4}{384} \\
&\leq \max_{\xi\in \Omega}u^{\left(4\right)}(\xi) \frac{5}{576}\delta h^{4}=\mathcal{O}\left(h^4\right)\delta,
\end{split}
\end{equation*}
with $\iota_{i},\zeta_{i},\varsigma_{i},\widetilde{\iota}_{i},\widetilde{\zeta}_{i},\widetilde{\varsigma}_{i}\in\Omega$.
The proof is completed.
\end{proof}
\begin{lemma}\label{lemma3.7}
Let  $\delta=\widetilde{\delta}=r h$. Let $u_{Q}(y)$ and $u_{C}(y)$ be defined by \eqref{2.8} and \eqref{3.1}, respectively. Then
\begin{equation*}
\left|\int^{x_{\frac{i}{2}}+\delta}_{x_{\frac{i}{2}}-\delta} u_{C}(y)-u_{Q}(y)dy\right| =\mathcal{O}\left(h^4\right)\delta.
\end{equation*}
\end{lemma}
\begin{proof}
According to Lemmas \ref{lemma3.3} - \ref{lemma3.6}, we obtain the remainder
\begin{equation*}
\left|\int^{x_{\frac{i}{2}}+\delta}_{x_{\frac{i}{2}}-\delta} u_{C}(y)-u_{Q}(y) dy\right| = \left|Q_{l}+Q_{\frac{i}{2}}+Q_{r}\right|=\mathcal{O}\left(h^4\right)\delta.
\end{equation*}
The proof is completed.
\end{proof}

\begin{lemma}\label{lemma3.8}
Let $\delta=\widetilde{\delta}=r h$ and $u_{C}(y)$ be defined by \eqref{3.1}. Then
\begin{equation*}
\int^{x_{\frac{i}{2}}+\delta}_{x_{\frac{i}{2}}-\delta} u(y)-u_{C}(y) dy =\mathcal{O}\left(h^4\right)\delta.
\end{equation*}
\end{lemma}
\begin{proof}
According to \eqref{3.1} and Taylor series expansion \cite[p.\,158]{Aikinson:2009}, we have $$u(y)-u_{C}(y)=\frac{u^{(4)}(\xi_{j})}{4!}\left(y-x_{j-1}\right)\left(y-x_{j-3/4}\right)\left(y-x_{j-\frac{1}{2}}\right)\left(y-x_{j}\right),~y\in[x_{j-1},x_{j}],$$
for some $\xi_{j}\in(x_{j-1},x_{j})$ depending on $y$.
It yields
\begin{equation*}
\left|\int^{x_{\frac{i}{2}}+\delta}_{x_{\frac{i}{2}}-\delta}u(y)-u_{C}(y)dy\right|
\leq\frac{h^4}{24}\max_{\xi\in\Omega}\left|u^{(4)}({\xi})\right|\int^{x_{\frac{i}{2}}+\delta}_{x_{\frac{i}{2}}-\delta}1dy=\mathcal{O}\left(h^4\right)\delta.
\end{equation*}
The proof is completed.
\end{proof}
\begin{lemma}\label{lemma3.9}
Let $\delta=\widetilde{\delta}=r h\geq h$ and $u_{Q}(y)$ be defined by \eqref{2.8}. Then
\begin{equation*}
\left|\mathcal{L}_{\widetilde{\delta}} u(x_{\frac{i}{2}})-\mathcal{L}_{\widetilde{\delta},h} u(x_{\frac{i}{2}})\right|
=\frac{3}{\widetilde{\delta}^{3}}\left|\int^{x_{\frac{i}{2}}+\widetilde{\delta}}_{x_{\frac{i}{2}}-\widetilde{\delta}} u(y)-u_{Q}(y) dy\right| =\mathcal{O}\left(h^4\right)\widetilde{\delta}^{-2}.
\end{equation*}
\end{lemma}
\begin{proof}
According to \eqref{2.6}, \eqref{2.9}, Lemmas \ref{lemma3.7}-\ref{lemma3.8} and the triangle inequality, the desired result is obtained.
\end{proof}

\subsection{Local truncation error for nonsymmetric indefinite systems \eqref{ad2.22} with $\delta=\widetilde{\delta}\pm r_0 h$, $0<r_0 <1$}
Since the horizon of the material is a physical property of the material of the finite bar and should be independent of the computational mesh size $h$ \cite{WT:7730--7738}.
In fact, it is difficult to require that $\delta$ is set as the grid point in the multidimensional space.
It is natural to consider the nonlocal model when $\delta$ is  not set as a grid point.
\begin{lemma}\label{lemma3.9999}
Let $\delta< h$. Then for the quadrature rule \eqref{2.9}, it holds that
\begin{equation*}
\left|\mathcal{L}_{\delta} u(x_{\frac{i}{2}})-\mathcal{L}_{\widetilde{\delta},h} u(x_{\frac{i}{2}})\right| =\mathcal{O}\left(h^2\right).
\end{equation*}
\end{lemma}
\begin{proof}
From \eqref{2.3} and (\ref{2.5}) with $\delta< h$, we obtain
\begin{equation*}
\mathcal{L}_\delta u(x)= C_\delta u''(x) + \mathcal{O} \left(\int_{0}^\delta z^4\gamma_\delta(z)dz\right)=  u''(x) +\mathcal{O}\left(\delta^2\right)\leq  u''(x) +\mathcal{O}\left(h^2\right),
\end{equation*}
and
\begin{equation*}
\mathcal{L}_{\widetilde{\delta}}u(x)=C_{\widetilde{\delta}}u''(x)+\mathcal{O}\left(\int_{0}^{\widetilde{\delta}}z^4\gamma_{\widetilde{\delta}}(z)dz\right)
=u''(x)+\mathcal{O}\left(\widetilde{\delta}^2\right)=u''(x)+\mathcal{O}\left(h^2\right).
\end{equation*}
According to the above equations and Lemma \ref{lemma3.9}, it implies that
\begin{equation*}
\left|\mathcal{L}_{\delta}u(x_{\frac{i}{2}})-\mathcal{L}_{\widetilde{\delta},h}u(x_{\frac{i}{2}})\right|\leq \left|\mathcal{L}_{\delta}u(x_{\frac{i}{2}})-\mathcal{L}_{\widetilde{\delta}}u(x_{\frac{i}{2}})\right|
+\left|\mathcal{L}_{\widetilde{\delta}} u(x_{\frac{i}{2}})-\mathcal{L}_{\widetilde{\delta},h} u(x_{\frac{i}{2}})\right|
=\mathcal{O}(h^2).
\end{equation*}
The proof is completed.
\end{proof}

\begin{lemma}\label{theorem3.11}
Let $\delta=\mathcal{O}\left(h^\beta\right)$, $\beta\geq 0$. Then
\begin{equation*}
R_\frac{i}{2}:=\left|\mathcal{L}_{{{\delta}}} u(x_{\frac{i}{2}})-\mathcal{L}_{{\widetilde{\delta}},h} u(x_{\frac{i}{2}})\right|
\!=\!\left\{ \begin{split}
\mathcal{O}\left(h^{\min\{2,1+\beta\}}\right)    ~             &  ~ {\rm if}~ \delta  {\rm ~is ~not~ set~as~  the ~ grid~ point},\\
\mathcal{O}\left(h^{\max\left\{2,4-2\beta\right\}}\right)      &  ~ {\rm if}~ \delta  {\rm ~is~ set~as~ the ~ grid~ point}.
\end{split}
\right.
\end{equation*}
\end{lemma}
\begin{proof}
We can rewrite the general horizon parameter $\delta$ in \eqref{2.7} as the general form $\delta=ch^\beta$, $\beta\geq0$, $c>0$.
We prove the desired results via the following three cases.

We first consider the case $\beta\in [0,1]$. Since $\delta=ch^\beta$, $0\leq \beta\leq 1$, $c>0$, it implies that $\widetilde{\delta}=\delta-r_0 h=\mathcal{O}\left(h^{\beta}\right).$
From Lemma \ref{lemma3.9}, it yields
\begin{equation}\label{3.6}
\left|\mathcal{L}_{\widetilde{\delta}} u(x_{\frac{i}{2}})-\mathcal{L}_{\widetilde{\delta},h} u(x_{\frac{i}{2}})\right|=\mathcal{O}\left(h^{4-2\beta}\right).
\end{equation}

From \eqref{2.2}, we have
\begin{equation}\label{3.7}
  \mathcal{L}_{\delta} u(x_{\frac{i}{2}})-\mathcal{L}_{\widetilde{\delta}} u(x_{\frac{i}{2}})=e_{1}+e_{2}
\end{equation}
with
\begin{equation*}
e_{1}=\int_{0}^{\widetilde{\delta}}\left[\gamma_{\delta}(z)-\gamma_{\widetilde{\delta}}(z)\right]\left[u(x_{\frac{i}{2}}+z)-2u(x_{\frac{i}{2}})+u(x_{\frac{i}{2}}-z)\right]dz,
\end{equation*}
and
\begin{equation*}
e_{2}=\int_{\widetilde{\delta}}^{\delta}\gamma_{\delta}(z)\left[u(x_{\frac{i}{2}}+z)-2u(x_{\frac{i}{2}})+u(x_{\frac{i}{2}}-z)\right]dz.
\end{equation*}

Case $\uppercase\expandafter{\romannumeral1}$: $\beta\in (0,1]$. Using \eqref{3.7}, \eqref{2.3} and Taylor series expansion, there exist
\begin{equation*}
\begin{split}
e_1&=\frac{\widetilde{\delta}^{3}-\delta^{3}}{\delta^{3}}u''(x_{\frac{i}{2}})
+\left(\frac{3}{\delta^3}-\frac{3}{\widetilde{\delta}^3}\right)\int_0^{\widetilde{\delta}}\frac{z^4}{4!}\left[u^{(4)}(\xi_{1})+u^{(4)}(\xi_{2})\right]dz;\\
e_2&=-\frac{\widetilde{\delta}^{3}-\delta^{3}}{\delta^{3}}u''(x_{\frac{i}{2}})+\frac{3}{\delta^3}\int^{\delta}_{\widetilde{\delta}} \frac{z^4}{4!}\left[u^{(4)}(\xi_{3})+u^{(4)}(\xi_{4})\right]dz
\end{split}
\end{equation*}
with $\xi_{1}\in (x_{\frac{i}{2}},x_{\frac{i}{2}}+\widetilde{\delta})$, $\xi_{2}\in (x_{\frac{i}{2}}-\widetilde{\delta}, x_{\frac{i}{2}})$ and $\xi_{3}\in (x_{\frac{i}{2}}+\widetilde{\delta},x_{\frac{i}{2}}+\delta)$, $\xi_{4}\in (x_{\frac{i}{2}}-\delta, x_{\frac{i}{2}}-\widetilde{\delta})$.
It implies that
\begin{equation*}
\left|e_{1}+e_{2}\right|\leq \frac{1}{20}\max_{\xi \in \Omega }\left|u^{(4)}(\xi)\right|\frac{\widetilde{\delta}^{3}-\delta^{3}} {\delta^{3}}\widetilde{\delta}^2
+ \frac{1}{20}\max_{\xi \in \Omega }\left|u^{(4)}(\xi)\right|\frac{\delta^{5}-\widetilde{\delta}^{5}}  {\delta^{3}}=\mathcal{O}\left(h^{1+\beta}\right).
\end{equation*}
Thus, using \eqref{3.6}, \eqref{3.7} and triangle inequality, for $\beta\in (0,1]$, we have
\begin{equation*}
\left|\mathcal{L}_{{\delta}} u(x_{\frac{i}{2}})-\mathcal{L}_{{\widetilde{\delta}},h} u(x_{\frac{i}{2}})\right|
\leq\left|\mathcal{L}_{\delta}u(x_{\frac{i}{2}})-\mathcal{L}_{\widetilde{\delta}}u(x_{\frac{i}{2}})\right|+\left|\mathcal{L}_{\widetilde{\delta}}u(x_{\frac{i}{2}})-\mathcal{L}_{\widetilde{\delta},h}u(x_{\frac{i}{2}})\right|
=\mathcal{O}\left(h^{1+\beta}\right).
\end{equation*}

Case $\uppercase\expandafter{\romannumeral2}$: $\beta=0$. From \eqref{3.7}, \eqref{2.3} and Taylor series expansion with $\xi_{5}\in(x_{\frac{i}{2}},x_{\frac{i}{2}}+z)$ and $\xi_{6}\in(x_{\frac{i}{2}}-z,x_{\frac{i}{2}})$, one has
\begin{equation*}
e_{1}=3\left(\frac{1}{\delta^{3}}-\frac{1}{\widetilde{\delta}^{3}}\right)\int_0^{\widetilde{\delta}}[u'(\xi_{5})-u'(\xi_{6})]z dz=\mathcal{O}\left(h\right),
\end{equation*}
since
\begin{equation*}
\frac{\widetilde{\delta}^{3}-\delta^{3}}{\delta^{3}\widetilde{\delta}^{3}} =\frac{1-\left(1+\frac{\delta_{0}h}{\widetilde{\delta}}\right)^{3}}{\delta^{3}}=\mathcal{O}\left(h\right).
\end{equation*}
Moreover, there exist  $\xi_{7}\in(x_{\frac{i}{2}},x_{\frac{i}{2}}+z)$ and $\xi_{8}\in(x_{\frac{i}{2}}-z,x_{\frac{i}{2}})$ such that
\begin{equation*}
e_{2}=\frac{3}{\delta^{3}}\int_{\widetilde{\delta}}^{\delta}[u'(\xi_{7})-u'(\xi_{8})]z dz \leq \frac{3}{\delta^{3}}2\max_{\xi\in\Omega} u'(\xi)\delta \int_{\widetilde{\delta}}^{\delta}1dz =\mathcal{O}\left(h\right).
\end{equation*}
Thus, using \eqref{3.6}, \eqref{3.7} and triangle inequality, for $\delta=c$, we have
\begin{equation*}
\left|\mathcal{L}_{{\delta}} u(x_{\frac{i}{2}})-\mathcal{L}_{\widetilde{\delta},h} u(x_{\frac{i}{2}})\right|
\leq\left|\mathcal{L}_{\delta}u(x_{\frac{i}{2}})-\mathcal{L}_{\widetilde{\delta}}u(x_{\frac{i}{2}})\right|
+\left|\mathcal{L}_{\widetilde{\delta}}u(x_{\frac{i}{2}})-\mathcal{L}_{\widetilde{\delta},h}u(x_{\frac{i}{2}})\right|
=\mathcal{O}\left(h\right).
\end{equation*}

Case $\uppercase\expandafter{\romannumeral3}$: $\beta>1$. It corresponds to $\delta=\widetilde{\delta}-r_0 h< h$ in \eqref{2.7}.
From Lemma \ref{lemma3.9999}, we have
\begin{equation*}
\left|\mathcal{L}_{{\delta}} u(x_{\frac{i}{2}})-\mathcal{L}_{{\widetilde{\delta}},h} u(x_{\frac{i}{2}})\right| =\mathcal{O}\left(h^2\right).
\end{equation*}
The proof is completed.
\end{proof}

\subsection{Perturbation error for symmetric positive definite system \eqref{2.17} and nonsymmetric indefinite  systems \eqref{ad2.22}}
We next prove the perturbation error  for symmetric systems \eqref{2.17} and nonsymmetric systems \eqref{ad2.22}.

\begin{lemma}
Let  $\widetilde{\delta}=rh$ with $r<\infty$.
Let the operations $\mathcal{L}_{{\widetilde{\delta}},h}u$ and $\mathcal{L}_{\widetilde{\delta},h}^S u$ be defined by \eqref{addd2.0012} and \eqref{2.11}, respectively.
Then
\begin{equation*}
\left|\mathcal{L}_{\widetilde{\delta},h} u(x_{i+\frac{1}{2}})-\mathcal{L}^S_{\widetilde{\delta},h} u(x_{i+\frac{1}{2}})\right| =\mathcal{O}\left(h^4\right)\widetilde{\delta}^{-2},  ~~0\leq i\leq N.
\end{equation*}
\end{lemma}
\begin{proof}
According to  \eqref{addd2.0012} and \eqref{2.11}, we have
\begin{equation*}
\begin{split}
\mathcal{L}_{\widetilde{\delta},h} u(x_{i+\frac{1}{2}})-\mathcal{L}^{S}_{\widetilde{\delta},h} u(x_{i+\frac{1}{2}})
 & =\frac{3}{{\widetilde{\delta}}^{3}}h \left[\sum_{m=1}^{r}{p_m}\left(u(x_{i+\frac{1}{2}+m})+u(x_{i+\frac{1}{2}-m})\right)+p_{0}u(x_{i+\frac{1}{2}})\right] \\
 & \quad +\frac{3}{{\widetilde{\delta}}^{3}}h\sum_{m=1}^{r+1}{q_{m-1}}\left(u(x_{i+m})+u(x_{i+1-m}) \right)
\end{split}
\end{equation*}
with
\begin{equation*}
\begin{split}
&   p_{m}=\frac{1}{3}, \quad 0\leq m\leq r-1, \qquad  p_{r}=\frac{1}{6}, \\
&   q_{m}=-\frac{1}{3}, \quad 0\leq m\leq r-2, \qquad  q_{r-1}=-\frac{7}{24}, \qquad  q_{r}=-\frac{1}{24}.
\end{split}
\end{equation*}
Using Taylor series expansion at the point $x_{i+\frac{1}{2}}$, we have
\begin{equation*}
\begin{split}
 \mathcal{L}_{{\widetilde{\delta}},h}u(x_{i+\frac{1}{2}})-\mathcal{L}^{S}_{{\widetilde{\delta}},h}u(x_{i+\frac{1}{2}})
& =\frac{3}{{\widetilde{\delta}}^{3}} h \sum_{m=1}^{r} p_{m}\left(\frac{m^{4}}{4!}h^{4}\left(u^{(4)}(\xi_{i,m})+u^{(4)}(\widetilde{\xi}_{i,m})\right)\right)\\
 &   \quad+\frac{3}{{\widetilde{\delta}}^{3}} h\sum_{m=1}^{r+1}{q_{m-1}}\frac{\left(m-\frac{1}{2}\right)^{4}}{4!}h^{4} \left( u^{(4)}(\eta_{i,m})+u^{(4)}(\widetilde{\eta}_{i,m})\right) ,
\end{split}
\end{equation*}
where we use
$$\frac{3}{{\widetilde{\delta}}^{3}} h \left[\sum_{m=1}^{r}  2p_{m}  +p_{0} + \sum_{m=1}^{r+1} 2q_{m-1}\right] u(x_{i+\frac{1}{2}})=0,$$
and
$$\frac{3}{{\widetilde{\delta}}^{3}} h\left[\sum_{m=1}^{r}{p_m} m^{2} +\sum_{m=1}^{r+1}{q_{m-1}}\left(m-\frac{1}{2}\right)^{2} \right]h^{2}u''(x_{i+\frac{1}{2}})=0.$$

Taking $M:=\max\limits_{x\in\Omega} u^{(4)}(x)$, it is easy to check that
\begin{equation*}
\begin{split}
\left|\mathcal{L}_{{\widetilde{\delta}},h} u(x_{i+\frac{1}{2}})-\mathcal{L}^{S}_{{\widetilde{\delta}},h} u(x_{i+\frac{1}{2}})\right|
&\leq \frac{M}{{4\widetilde{\delta}}^{3}} h^5 \left[\sum_{m=1}^{r} |p_{m}| m^{4}  +\sum_{m=1}^{r+1} |q_{m-1}| \left(m-\frac{1}{2}\right)^{4} \right]\\
&= \frac{M}{{4\widetilde{\delta}}^{3}} h^5 \left(\frac{2}{15}r^5+\frac{2}{9}r^3+\frac{29}{720}r \right) =C_r\frac{M}{4} h^4 \widetilde{\delta}^{-2}
\end{split}
\end{equation*}
with $\widetilde{\delta}=rh$ and
\begin{equation}\label{b3.8}
C_r=\frac{2}{15}r^4+\frac{2}{9}r^2+\frac{29}{720}.
\end{equation}
The proof is completed.
\end{proof}

Now we consider the case $r\rightarrow \infty$. Note that in this case, the estimate $C_r\rightarrow \infty$ in \eqref{b3.8}, which leads to
$$C_r\frac{M}{4} h^4 \widetilde{\delta}^{-2}\geq \frac{2}{15}\frac{M}{4}\widetilde{\delta}^{2}=\mathcal{O}\left(1\right)~{\rm if }~\widetilde{\delta}=\mathcal{O}\left(1\right).$$
Therefore, we need to look for an estimate of other form with $r\rightarrow \infty$.  The technique based on Taylor series expansions, which requires the  smooth enough or say $u\in C_c^\infty$ on the solution  \cite[pp.\,91-92]{Du:2017}.
This condition seems more theoretically rather than   practically as explained in the following result.
\begin{lemma}\label{theorem3.12222}
Let the operations $\mathcal{L}_{{\widetilde{\delta}},h}u$ and $\mathcal{L}_{\widetilde{\delta},h}^S u$ be defined by \eqref{ad2.22} and \eqref{2.17}, respectively.
Then
\begin{equation*}
\left|\mathcal{L}_{\widetilde{\delta},h} u(x_{i+\frac{1}{2}})-\mathcal{L}^S_{\widetilde{\delta},h} u(x_{i+\frac{1}{2}})\right| =\mathcal{O}\left(h^4\right)\widetilde{\delta}^{-2}, \quad i=0,1,2,\ldots,N.
\end{equation*}
\end{lemma}
\begin{proof}
According to \eqref{addd2.0012} and \eqref{2.11}, we have
\begin{equation*}
\begin{split}
\mathcal{L}_{\widetilde{\delta},h} u(x_{i+\frac{1}{2}})-\mathcal{L}^{S}_{\widetilde{\delta},h} u(x_{i+\frac{1}{2}})
& =\frac{3}{\widetilde{\delta}^{3}}h \left[\sum_{m=1}^{r}{p_{m}}\left(u(x_{i+\frac{1}{2}+m})+u(x_{i+\frac{1}{2}-m})\right)+p_{0}u(x_{i+\frac{1}{2}})\right]\\
& \quad +\frac{3}{\widetilde{\delta}^{3}}h\sum_{m=1}^{r+1}{q_{m-1}}\left(u(x_{i+m})+u(x_{i+1-m}) \right)\\
\end{split}
\end{equation*}
with
\begin{equation*}
\begin{split}
&   p_{m}=\frac{1}{3}, \quad 0\leq m\leq r-1, \qquad  p_{r}=\frac{1}{6}, \\
&   q_{m}=-\frac{1}{3}, \quad 0\leq m\leq r-2, \qquad  q_{r-1}=-\frac{7}{24}, \qquad  q_{r}=-\frac{1}{24}.
\end{split}
\end{equation*}
Using Taylor series expansion, we have
\begin{equation*}
\begin{split}
u(x_{i+\frac{1}{2}+m})+u(x_{i+\frac{1}{2}-m}) & =  \sum^{\infty}_{n=0}2 \frac{m^{2n}}{(2n)!} h^{2n}u^{(2n)}(x_{i+\frac{1}{2}}),\\
u(x_{i+m})+u(x_{i+1-m})                       & =  \sum^{\infty}_{n=0}2 \frac{\left(m-\frac{1}{2}\right)^{2n}}{(2n)!} h^{2n}u^{(2n)}(x_{i+\frac{1}{2}}).
\end{split}
\end{equation*}
According to the above equations, it leads to
\begin{equation*}
\begin{split}
\mathcal{L}_{\widetilde{\delta},h}u(x_{i+\frac{1}{2}})-\mathcal{L}^{S}_{\widetilde{\delta},h}u(x_{i+\frac{1}{2}})
& =\frac{3}{\widetilde{\delta}^{3}}h\sum^{\infty}_{n=2}\left[\sum_{m=1}^{r}\frac{1}{3}\left( m^{2n} -2\left(m-\frac{1}{2}\right)^{2n}+(m-1)^{2n} \right)  \right. \\
& \quad \left. +\frac{1}{12}\left(r-\frac{1}{2}\right)^{2n}-\frac{1}{12}\left(r+\frac{1}{2}\right)^{2n}\right] \frac{h^{2n}}{(2n)!}  u^{(2n)}(x_{i+\frac{1}{2}}),
\end{split}
\end{equation*}
where we use
\begin{equation*}
\sum_{m=1}^{r}2p_{m}+p_{0} +\sum_{m=1}^{r+1}2 q_{m-1}=0~~{\rm and}~~\sum_{m=1}^{r} {p_m} m^{2} +\sum_{m=1}^{r+1}q_{m-1}\left(m-\frac{1}{2}\right)^{2}=0.
\end{equation*}
Using binomial theorem, we have
\begin{equation*}
m^{2n} -2\left(m-\frac{1}{2}\right)^{2n}+(m-1)^{2n} = 2 \left(m-\frac{1}{2}\right)^{2n} \sum^{n}_{k=1} \binom{2n}{2k} \left( \frac{1}{2m-1} \right)^{2k},
\end{equation*}
and
\begin{equation*}
\left(r-\frac{1}{2}\right)^{2n}-\left(r+\frac{1}{2}\right)^{2n} = -2r^{2n} \sum^{n}_{k=1}\binom{2n}{2k-1}\left( \frac{1}{2r} \right)^{2k-1}.
\end{equation*}
Then we obtain
\begin{equation*}
\begin{split}
\mathcal{L}_{\widetilde{\delta},h} u(x_{i+\frac{1}{2}})-\mathcal{L}^{S}_{\widetilde{\delta},h} u(x_{i+\frac{1}{2}})
& =\frac{3}{\widetilde{\delta}^{3}}h\sum^{\infty}_{n=2}\left[\sum^{n}_{k=1} \left(\frac{2}{3} \binom{2n}{2k}\sum_{m=1}^{r}2^{-2n}(2m-1)^{2n-2k} \right. \right. \\
& \left. \left. \qquad -\frac{1}{6}\binom{2n}{2k-1}2^{-2k+1}r^{2n-2k+1}\right)\right] \frac{h^{2n}}{(2n)!}  u^{(2n)}(x_{i+\frac{1}{2}}).
\end{split}
\end{equation*}
According to Faulhaber's formula \cite{Knuth:277--294}, there exists
\begin{equation*}
\sum^{n}_{k=1}k^{p}=\frac{1}{p+1}\sum^{p}_{k=0}(-1)^{k}\binom{p+1}{k}B_{k}n^{p+1-k},
\end{equation*}
where $B_{k}$ is Bernoulli numbers \cite[p.\,22]{Koblitz:1996}
\begin{equation*}
B_{0}=1, ~~ B_{1}=-\frac{1}{2}, ~~ B_{2k+1}=0 ~~ {\rm and} ~~ B_{2k}=2(-1)^{k+1}\frac{\zeta(2k) (2k)!}{(2\pi)^{2k}} ~~  {\rm for}~ k=1,2,\ldots
\end{equation*}
and $\zeta(2k)$ is the Riemann zeta function \cite[p.\,27]{Koblitz:1996}
\begin{equation*}
\zeta(s)=\sum^{\infty}_{n=1}\frac{1}{n^{s}} \quad (s>1).
\end{equation*}

Using the above equations, we have
\begin{equation*}
\begin{split}
&\mathcal{L}_{{\widetilde{\delta}},h} u(x_{i+\frac{1}{2}})-\mathcal{L}^{S}_{{\widetilde{\delta}},h} u(x_{i+\frac{1}{2}})\\
& =\frac{3}{{\widetilde{\delta}}^{3}}h\sum^{\infty}_{n=2}\left[\sum^{n}_{k=1} \left(\frac{2}{3} \binom{2n}{2k} \sum_{j=0}^{2n-2k} \frac{\left( 2^{1-j}-1 \right)2^{-2k}}{2n-2k+1}  \binom{2n-2k+1}{j}(-1)^{j} B_{j}  r^{2n-2k+1-j} \right. \right. \\
&\qquad \qquad \qquad \qquad \left. \left. -\frac{1}{6}\binom{2n}{2k-1}2^{-2k+1}r^{2n-2k+1}\right)\right] \frac{h^{2n}}{(2n)!}  u^{(2n)}(x_{i+\frac{1}{2}})\\
& = \uppercase\expandafter{\romannumeral1} + \uppercase\expandafter{\romannumeral2} + \uppercase\expandafter{\romannumeral3} + \uppercase\expandafter{\romannumeral4},
\end{split}
\end{equation*}
where
\begin{equation*}
\begin{split}
&\uppercase\expandafter{\romannumeral1}=\frac{3}{\widetilde{\delta}^{3}}h\sum^{\infty}_{n=2}-\frac{1}{192}\frac{(rh)^{2n-3}}{(2n-3)!} h^{3} u^{(2n)}(x_{i+\frac{1}{2}});\\
&\uppercase\expandafter{\romannumeral2}=\frac{3}{\widetilde{\delta}^{3}}h\sum^{\infty}_{n=3}\sum_{k=2}^{n-1}\frac{1}{12}\frac{\left(2^{1-2k}-1\right)}{(2n-2k-1)!}\frac{B_{2k}}{(2k)!}r^{2n-2k-1}h^{2n}
u^{(2n)}(x_{i+\frac{1}{2}});\\
&\uppercase\expandafter{\romannumeral3}=\frac{3}{\widetilde{\delta}^{3}}h\sum^{\infty}_{n=3}\sum_{k=2}^{n-1} \frac{1}{12} \frac{2^{-2k}}{(2n-2k-1)!}\frac{-2k}{(2k+2)!}r^{2n-2k-1} h^{2n}  u^{(2n)}(x_{i+\frac{1}{2}});\\
&\uppercase\expandafter{\romannumeral4}=\frac{3}{\widetilde{\delta}^{3}}h\sum^{\infty}_{n=3}\sum^{n-1}_{k=2} \frac{2}{3} \frac{1}{(2k)!} \sum_{j=1}^{n-k} \frac{\left( 2^{-2k+1-2j}-2^{-2k} \right)}{(2n-2k+1-2j)!} \frac{B_{2j}}{(2j)!}  r^{2n-2k+1-2j} h^{2n} u^{(2n)}(x_{i+\frac{1}{2}}).
\end{split}
\end{equation*}
Since
\begin{equation*}
\left|\zeta(2k)\right|=\sum^{\infty}_{n=1}\frac{1}{n^{2k}} \leq \sum^{\infty}_{n=1}\frac{1}{n^{2}} \leq 1+ \sum^{\infty}_{n=2}\frac{1}{(n-1)n} =1+ \sum^{\infty}_{n=2}\left(\frac{1}{n-1} - \frac{1}{n}\right) \leq 2,
\end{equation*}
which yields
\begin{equation*}
\left|\frac{B_{2k}}{(2k)!}\right|=\left|2(-1)^{k+1}\frac{\zeta(2k)}{(2\pi)^{2k}}\right|\leq \frac{4}{(2\pi)^{2k}}.
\end{equation*}

Taking $M:=\max\limits_{n} u^{(n)}(x_{i+\frac{1}{2}})$ with  $\widetilde{\delta}=rh$, we have
\begin{equation*}
\left| \uppercase\expandafter{\romannumeral1} \right| \leq  \frac{3}{\widetilde{\delta}^{3}}h\sum^{\infty}_{n=2}\frac{1}{192}  \frac{\widetilde{\delta}^{2n-3}}{(2n-3)!} h^{3} M
=\frac{1}{64}\frac{1}{\widetilde{\delta}^{3}}M  h^{4} \sum^{\infty}_{n=0} \frac{\widetilde{\delta}^{2n+1}}{(2n+1)!} \leq \frac{1}{64}\frac{1}{\widetilde{\delta}^{2}}M e^{\widetilde{\delta}} h^{4};
\end{equation*}
\begin{equation*}
\begin{split}
\left| \uppercase\expandafter{\romannumeral2} \right|
& \leq \frac{3}{\widetilde{\delta}^{3}} M h\sum^{\infty}_{n=3}\sum_{k=2}^{n-1}  \frac{1}{12} \frac{\widetilde{\delta}^{2n-2k-1}}{(2n-2k-1)!} \frac{4}{(2\pi)^{2k}}   h^{2k+1} \\
& = \frac{1}{(2\pi)^{4}} \frac{1}{\widetilde{\delta}^{3}} M h^{6} \sum_{k=0}^{\infty} \frac{ h^{2k}}{(2\pi)^{2k}} \sum^{\infty}_{n=0} \frac{\widetilde{\delta}^{2n+1}}{(2n+1)!}
  \leq \frac{1}{8}\frac{1}{\pi^{4}}\frac{1}{\widetilde{\delta}^{2}}M e^{\widetilde{\delta}} h^{6};
\end{split}
\end{equation*}
and
\begin{equation*}
\begin{split}
\left| \uppercase\expandafter{\romannumeral3} \right|
& \leq \frac{1}{4} \frac{1}{\widetilde{\delta}^{3}}M h\sum^{\infty}_{n=3}\sum_{k=2}^{n-1} \frac{\widetilde{\delta}^{2n-2k-1}}{(2n-2k-1)!}\frac{2k}{(2k+2)!}2^{-2k} h^{2k+1} \\
& = \frac{1}{64} \frac{1}{\widetilde{\delta}^{3}}M h^{6} \sum_{k=0}^{\infty}\frac{2k+4}{(2k+2+4)!} \frac{h^{2k}}{2^{2k}}\sum^{\infty}_{n=0} \frac{\widetilde{\delta}^{2n+1}}{(2n+1)!}
  \leq \frac{1}{32} \frac{1}{\widetilde{\delta}^{2}}M e^{\widetilde{\delta}} h^{6};
\end{split}
\end{equation*}
\begin{equation*}
\begin{split}
\left| \uppercase\expandafter{\romannumeral4} \right|
& \leq 2\frac{1}{\widetilde{\delta}^{3}}M h\sum^{\infty}_{n=3}\sum^{n-1}_{k=2} \sum_{j=1}^{n-k}\frac{2^{-2k}}{(2k)!} \frac{\widetilde{\delta}^{2n-2k+1-2j}}{(2n-2k+1-2j)!}  \frac{4}{(2\pi)^{2j}}   h^{2k+2j-1} \\
& = \frac{1}{8}\frac{1}{\pi^{2}}\frac{1}{\widetilde{\delta}^{3}} M h^{6}\sum^{\infty}_{k=0}\frac{1}{(2k+4)!} \frac{ h^{2k}}{2^{2k}} \sum_{j=0}^{\infty} \frac{h^{2j}}{(2\pi)^{2j}} \sum^{\infty}_{n=0}
    \frac{\widetilde{\delta}^{2n+1}}{(2n+1)!}
 \leq \frac{1}{2}\frac{1}{\pi^{2}}\frac{1}{\widetilde{\delta}^{2}}  M e^{\widetilde{\delta}} h^{6}.
\end{split}
\end{equation*}
Hence, we obtain
\begin{equation*}
\left|\mathcal{L}_{\widetilde{\delta},h} u(x_{i+\frac{1}{2}})-\mathcal{L}^{S}_{{\widetilde{\delta}},h} u(x_{i+\frac{1}{2}})\right|
\leq \frac{1}{64}M e^{\widetilde{\delta}} h^{4}\widetilde{\delta}^{-2} + \left(\frac{1}{8}\frac{1}{\pi^{4}} + \frac{1}{32}+\frac{1}{2}\frac{1}{\pi^{2}} \right) M e^{\widetilde{\delta}} h^{6} \widetilde{\delta}^{-2}.
\end{equation*}
The proof is completed.
\end{proof}

\begin{theorem}\label{theorem3.13}
Let $\delta=\mathcal{O}\left(h^\beta\right)$, $\beta\geq 0$. Then
\begin{equation*}
R^S_\frac{i}{2}:=\left|\mathcal{L}_{{{\delta}}} u(x_{\frac{i}{2}})-\mathcal{L}^S_{{\widetilde{\delta}},h} u(x_{\frac{i}{2}})\right|
=\left\{\begin{split}
\mathcal{O}\left(h^{\min\{2,1+\beta\}}\right)   ~              &  ~~ {\rm if}~ \delta  {\rm ~is ~not~  the ~ grid~ point},\\
\mathcal{O}\left(h^{\max\left\{2,4-2\beta\right\}}\right)      &  ~~ {\rm if}~ \delta  {\rm ~is  ~ the ~ grid~ point}.
\end{split}\right.
\end{equation*}
\end{theorem}
\begin{proof}
According to Lemmas \ref{theorem3.11}, \ref{theorem3.12222} and the triangle inequality, the desired result is obtained.
\end{proof}

\section{Stability and Convergence Analysis}
In this section, the detailed proof of the global error for the nonlocal models \eqref{2.1} with a general horizon parameter are provided.
We first introduce some lemmas, which will be used later.
\begin{lemma} \label{lemma4.1}
Let matrix $A^S_{\widetilde{\delta},h}$ be defined by \eqref{2.19}.
Then $A^S_{\widetilde{\delta},h}$ is a diagonally dominant symmetric matrix with positive entries on the diagonal and nonpositive off-diagonal entries.
\end{lemma}
\begin{proof}
From \eqref{2.200000}, we have
\begin{equation*}
a_{0}<0, \quad a_{m}>0,~~1\leq  m\leq r \quad {\rm and} \quad a_{m+\frac{1}{2}}>0,~~0\leq m\leq r-1.
\end{equation*}
Then $A^S_{\widetilde{\delta},h}$ is a symmetric matrix with positive entries on the diagonal and nonpositive off-diagonal entries;
and $A^S_{\widetilde{\delta},h}$ is diagonally dominant symmetric matrix, since from \eqref{2.16}, there exists
\begin{equation*}
a_0+2\sum_{m=1}^{r}a_m+2\sum_{m=0}^{r-1}a_{m+\frac{1}{2}}=0.
\end{equation*}
The proof is completed.
\end{proof}

Note that Lemma \ref{lemma4.1} does not guarantee that the matrix $A^S_{\widetilde{\delta},h}$ is neither nonsingular nor positive definite, since $A^S_{\widetilde{\delta},h}$ is reducible.
Hence, we need to prove it positive definite in further.

\begin{lemma} \label{lemma4.2}
Let matrix $A^S_{\widetilde{\delta},h}$ be defined by \eqref{2.19}.
Then $A^S_{\widetilde{\delta},h}$ is a symmetric positive definite matrix with positive entries on the diagonal and nonpositive off-diagonal entries.
Moreover, $\left(A^S_{\widetilde{\delta},h}\right)^{-1}$ is a positive matrix.
\end{lemma}
\begin{proof}
Let $L_{N}={\rm tridiag}(-1,2,-1)$ be the $N\times N$ one dimensional discrete Laplacian.
Let
\begin{equation}\label{4.1}
A_{\rm res}=A^S_{\widetilde{\delta},h}-\eta^h_{\widetilde{\delta}}a_1A_{\rm main}
 ~~{\rm with}~~A_{\rm main}=
\left[\begin{array}{cc}
L_{N} & O        \\
	O & L_{N+1}
\end{array}\right].
\end{equation}

From \eqref{2.16} and \eqref{4.1}, it shows that  $A_{\rm res}$ is a diagonally dominant symmetric matrix with positive entries on the diagonal and nonpositive off-diagonal entries.
Thus, $A_{\rm res}$ is a semi-positive matrix by Gerschgorin circle theorem \cite[p.\,21]{Varga:2000}. On the other hand, the matrix $A_{\rm main}$ is a positive matrix, since the determinant $\det A_{\rm main}=(N+1)(N+2)>0$.
The proof is completed.
\end{proof}

From Lemma \ref{lemma4.2}, we have the following discrete maximum principle.
\begin{proposition}\label{proposition4.3}
The collocation schemes \eqref{2.17} satisfy the discrete maximum principle:
\begin{equation*}
\begin{split}
f(x)\leq 0,~ for ~  x\in\Omega\Rightarrow \max_{i\in\mathcal{N}_{in}}u_i\leq \max_{i\in\mathcal{N}_{out}}{u_i};\\
f(x)\geq 0,~for ~ x\in\Omega\Rightarrow \min_{i\in\mathcal{N}_{in}}u_i\geq \min_{i\in\mathcal{N}_{out}}{u_i}.
\end{split}
\end{equation*}
\end{proposition}

\begin{lemma}\label{lemma4.4}
Let the discrete operator $\mathcal{L}^S_{\widetilde{\delta},h}$ be defined by \eqref{2.11}. Then
\begin{eqnarray*}
\left\|\left(-\mathcal{L}^S_{\widetilde{\delta},h}\right)^{-1}\right\|_\infty\leqslant\frac{1+4\delta(1+\delta)}{8C_{\widetilde{\delta}}}
~~{\rm with}~C_{\widetilde{\delta}}>0~~{\rm in}~\eqref{2.4}.
\end{eqnarray*}
\end{lemma}
\begin{proof}
The proof is based on the idea of \cite{Zhang:52--68}.
Let $v_{\delta}(x)=\frac{x(1-x)+\delta(1+\delta)}{2C_{\widetilde{\delta}}}\geq 0$. From \eqref{2.6}, we have
\begin{equation*}
-\mathcal{L}_{\widetilde{\delta}} v_\delta(x)  =-\int_{0}^{\widetilde{\delta}}\gamma_{\widetilde{\delta}}(z )\left[v_\delta(x+z)-2v_\delta(x)+v_\delta(x-z)\right]dz
=\frac{1}{C_{\widetilde{\delta}}}\int_{0}^{\widetilde{\delta }} z^2\gamma_{\widetilde{\delta}}(z)dz=1.
\end{equation*}
Let
$$v_{\delta,h}(x)=\left[v_{\delta}(x_1),v_{\delta}(x_2),\cdots,v_{\delta}(x_N),v_{\delta}(x_{\frac{1}{2}}),v_{\delta}(x_{\frac{3}{2}}),\cdots,v_{\delta}(x_{N+\frac{1}{2}})  \right]^T,$$
and
$$g_{\delta}(x)=\left[1,1,\cdots,1,1,1,\cdots,1  \right]^T.$$
It yields
\begin{equation*}
-\mathcal{L}_{\widetilde{\delta}}v_{\delta,h}(x)=g_{\delta}(x).
\end{equation*}

On the other hand, using \eqref{2.11}, \eqref{2.200000}, \eqref{2.16} and Taylor  expansions, there exists
\begin{equation}\label{ada4.2}
\begin{split}
-\mathcal{L}_{\widetilde{\delta},h}^S v_{\delta}\left(x_\frac{i}{2}\right)
&=-\eta^h_{\widetilde{\delta}} \left[\sum_{m=-r}^{r}{a_m}u(x_{\frac{i}{2}+m})+\sum_{m=-r}^{r-1}{a_{m+\frac{1}{2}}}u(x_{\frac{i}{2}+m+\frac{1}{2}})\right]\\
&=\eta^h_{\widetilde{\delta}} h^2\left(\sum_{m=1}^{r}a_{m}m^2+\sum_{m=0}^{r-1}a_{m+\frac{1}{2}}\left(m+\frac{1}{2}\right)^2\right)\\
&=\eta^h_{\widetilde{\delta}} h^2\left(3r^2-r+\sum_{m=1}^{r-1}6m^2 \right)=1.\\
\end{split}
\end{equation}

According to the above equations, we get
\begin{equation*}
-\mathcal{L}^S_{\widetilde{\delta},h}v_{\delta,h}(x)=-\mathcal{L}_{\widetilde{\delta}}v_{\delta,h}(x)=g_{\delta}(x).
\end{equation*}

Using Lemmas \ref{lemma4.1}, \ref{lemma4.2} and the discrete maximum principle in Proposition \ref{proposition4.3} with $v_{\delta}(x_j)\geqslant0$ for all $j\in\mathcal{N}_{out}$, we obtain
\begin{equation*}
\left\|\left(-\mathcal{L}^{S}_{{\widetilde{\delta}},h}\right)^{-1}\right\|_\infty=\left\|\left(-\mathcal{L}^{S}_{\widetilde{\delta},h}\right)^{-1}g_{\delta}\right\|_\infty=\|v_{\delta,h}\|_\infty
\leq \|v_{\delta}\|_\infty \leqslant\frac{1+4\delta(1+\delta)}{8C_{\widetilde{\delta}}}.
\end{equation*}
The proof is completed.
\end{proof}

\subsection{Error analysis for symmetric systems \eqref{2.17}}
From Lemma \ref{lemma4.4} and stability definition in \cite[p.19]{LeVeque:2007}, the stability of the discrete scheme \eqref{2.18} can be established immediately.
We now show the convergence behavior and error estimates with the general horizon parameter $\delta$.
\begin{theorem}\label{theorem4.5}
Let $\delta=\mathcal{O}\left(h^\beta\right)$, $\beta\geq 0$.
Let $u_{\delta,h}(x_{i})$ be the approximate solution of $u_{\delta}(x_{i})$ computed by the discretization scheme \eqref{2.18}. Then
\begin{equation*}
\left\|u_{\delta,h}(x_i)-u_\delta(x_i)\right\|_\infty
=\left\{\begin{split}
\mathcal{O}\left(h^{\min\{2,1+\beta\}}\right)    ~             &  ~~ {\rm if}~ \delta  {\rm ~is ~not~ set~as~ the ~ grid~ point},\\
\mathcal{O}\left(h^{\max\left\{2,4-2\beta\right\}}\right)      &  ~~ {\rm if}~ \delta  {\rm ~is~ set~as~  ~ the ~ grid~ point}.
\end{split}\right.
\end{equation*}
\end{theorem}
\begin{proof}
From theorem \ref{theorem3.13}, we can rewrite \eqref{2.1} as
$$-\mathcal{L}^{S}_{\widetilde{\delta},h}u_{\delta}(x_{i})=f_{\delta}(x_i)+R_{i}^{S}.$$
Subtracting \eqref{2.1} from \eqref{2.17}, it yields
\begin{eqnarray*}
-\mathcal{L}^{S}_{\widetilde{\delta},h}[u_{\delta,h}(x_{i})-u_\delta(x_{i})]=R_{i}^{S}.
\end{eqnarray*}
Thus, we have
 \begin{equation*}
\left\|u_{\delta,h}(x_{i})-u_\delta(x_{i})\right\|_\infty
\leqslant \left\|\left(-\mathcal{L}^{S}_{\widetilde{\delta},h}\right)^{-1}\right\|_\infty \left\|R_{i}^{S}\right\|_\infty.
\end{equation*}
According to Lemma \ref{lemma4.4} and Theorems \ref{theorem3.13}, the desired results are obtained.
\end{proof}

\subsection{Error analysis of AC scheme}
To connect the nonlocal problem \eqref{2.1} with its local limit, we also require that
\begin{equation}\label{4.2}
C_\delta\rightarrow C_0,~~f_\delta\rightarrow f_0,~~{\rm as}~~\delta\rightarrow 0.
\end{equation}
The solution of nonlocal problems \eqref{2.1} converges to the solution of the two-point boundary value problem \cite{Gunzburger:676--696} as  $\delta\rightarrow0$, namely,
\begin{equation}\label{4.3}
\left\{ \begin{split}
-C_0u_0''(x)=f_0(x),    \quad   &~in~ \Omega  \\
u_0(x)=g(x),      ~    \quad   &~on~ \partial \Omega=\{0\}\cup\{1\}
\end{split}\right.
\end{equation}
which is the classic diffusion problem.

Let us study the asymptotic compatibility of the collocation scheme \eqref{2.17}, i.e.,
\begin{equation}\label{4.4}
   -\mathcal{L}^S_{\widetilde{\delta},h} u_{\delta,h}(x_i) =f_0(x_i),~~{\rm as}~~\delta\rightarrow 0.
\end{equation}
\newtheorem{Definition}{Definition}
\begin{Definition} \cite{Tian:1641--1665,Zhang:52--68}
A family of convergent approximations $\{u_{\delta,h}\}$ defined by \eqref{4.4} is said to be asymptotically compatible to the solution $u_{0}$ defined by \eqref{4.3} if both $\delta\rightarrow0$ and $h\rightarrow0$, we have $u_{\delta,h}\rightarrow u_0$.
\end{Definition}
\begin{lemma}\label{lemma4.6}
Let $\delta\leq h$ with $\delta,h\rightarrow 0$.
Let $u_{\delta,h}$ and $u_0$ be the solution of \eqref{4.4} and \eqref{4.3}, respectively. Then it holds that
\begin{equation*}
\left\|u_{\delta,h}-u_0\right\|_\infty=\mathcal{O}\left(h^2\right)~~{\rm as}~~\delta,h\rightarrow 0.
\end{equation*}
\end{lemma}
\begin{proof}
According to \eqref{2.11}, \eqref{2.22222} and Taylors series expansion, it yields
\begin{equation*}
\begin{split}
\mathcal{L}^S_{\widetilde{\delta},h}u_0(x_i)
&=\eta^h_{\widetilde{\delta}}\left[u_0\left(x_{i-1}\right)+u_0\left(x_{i+1}\right)-10u_{0}\left(x_i\right)+4\left(u_0\left(x_{i-\frac{1}{2}}\right)+u_0\left(x_{i+\frac{1}{2}}\right)\right)\right]\\
&=u_0''(x_i)+\left(\sum_{l=1}^{\infty}\frac{h^{2l}}{(2l+2)!}+\sum_{l=1}^{\infty}\frac{(\frac{h}{2})^{2l}}{(2l+2)!}\right)u_0^{(2l+2)}(x_i)\\
&=u_0''(x_i)+\mathcal{O}\left(h^2\right)\rightarrow C_0u_0''(x_i)~~{\rm as}~~\delta,h\rightarrow 0,
\end{split}
\end{equation*}
with $\eta^h_{\widetilde{\delta}}=\frac{h}{2\widetilde{\delta}^{3}}=\frac{1}{2h^{2}}$ and $C_\delta= C_0=1$. From \eqref{4.3}, \eqref{4.4} and Lemma \ref{lemma4.4}, we have
\begin{eqnarray*}
\left\|u_{\delta,h}-u_0\right\|_\infty \leqslant \left\|\left(-\mathcal{L}^S_{\widetilde{\delta},h}\right)^{-1}\right\|_\infty \left\|f_0+\mathcal{L}^S_{\widetilde{\delta},h}u_0\right\|_\infty
\leqslant \left\| \left(-\mathcal{L}^S_{\widetilde{\delta},h}\right)^{-1}\right\|_\infty \left\|\mathcal{L}^S_{\widetilde{\delta},h}u_0-u_0{''}\right\|_\infty.
\end{eqnarray*}
The proof is completed.
\end{proof}

\begin{lemma}\label{lemma4.7}
Let $\delta\geq h$ with $\delta,h\rightarrow 0$.
Let $u_{\delta,h}$ and $u_0$ be the solution of \eqref{4.4} and \eqref{4.3}, respectively. Then  it holds that
\begin{equation*}
\left\|u_{\delta,h}-u_{0} \right\|_\infty=\mathcal{O}\left(\delta^2\right)~~{\rm as}~~\delta,h\rightarrow 0.
\end{equation*}
\end{lemma}
\begin{proof}
From \eqref{2.11}, \eqref{2.200000}, \eqref{2.16} and the Taylors series expansion, it yields
\begin{equation}\label{4.5555}
\begin{split}
\mathcal{L}^S_{\widetilde{\delta},h}u_0(x_i)
=&\sum_{l=0}^{\infty}C_l^r(h)u_0^{(2l+2)}(x_i)+\eta^h_{\widetilde{\delta}}\left(2\sum_{m=1}^{r}a_m+2\sum_{m=0}^{r-1}a_{m+\frac{1}{2}}+a_0\right)u(x_i)\\
=&\sum_{l=0}^{\infty}C_l^r(h)u_0^{(2l+2)}(x_i)
=C_0^r(h)u_0''(x_i)+\sum_{l=1}^{\infty}C_l^r(h)u_0^{(2l+2)}(x_i)
\end{split}
\end{equation}
with
\begin{equation*}
C_l^r(h)=2\eta^h_{\widetilde{\delta}}\sum_{m=1}^{r}a_m\frac{(mh)^{2l+2}}{(2l+2)!}+2\eta^h_{\widetilde{\delta}}\sum_{m=0}^{r-1}a_{m+\frac{1}{2}}\frac{\left[\left(m+\frac{1}{2}\right)h\right]^{2l+2}}{(2l+2)!}.
\end{equation*}

We next prove $C_0^r(h)=C_\delta$ in \eqref{4.2}. Choosing $u_0(z)=z^2$ with $x=0$ in \eqref{2.6}, we get
\begin{eqnarray*}
\mathcal{L}_{\widetilde{\delta}} u_0(0)=\mathcal{L}_{\widetilde{\delta}} u_0(x) \big|_{x=0}=2\int_{0}^{\widetilde{\delta}} z^2\gamma_{\widetilde{\delta}}(z)dz=2C_{\widetilde{\delta}}
=2\int_{0}^{{\delta}} z^2\gamma_{{\delta}}(z)dz=2C_{{\delta}}=2.
\end{eqnarray*}
Similarly, taking  $x_i=0$ in \eqref{2.11} and using \eqref{ada4.2}, it leads to
 \begin{eqnarray*}
\mathcal{L}^S_{\widetilde{\delta},h} u_0(0)
=2\eta^h_{\widetilde{\delta}}\sum_{m=1}^{r}a_m(mh)^2+2\eta^h_{\widetilde{\delta}}\sum_{m=0}^{r-1}a_{m+\frac{1}{2}}\left[\left(m+\frac{1}{2}\right)h\right]^2\!\!\!=2C_0^r(h)=2.
\end{eqnarray*}
Hence, we have
$$C_0^r(h)=C_\delta=1.$$

On the other hand, there exists
\begin{eqnarray*}
\begin{split}
\left| C_l^r(h)\right|
\leq &\frac{2}{(2l+2)!}\left\{\eta^h_{\widetilde{\delta}}\sum_{m=1}^{r}a_m(rh)^{2l}(mh)^2+\eta^h_{\widetilde{\delta}}\sum_{m=0}^{r-1}a_{m+\frac{1}{2}}(rh)^{2l}\left[\left(m+\frac{1}{2}\right)h\right]^{2}\right\}\\
=  &\frac{2(rh)^{2l}}{(2l+2)!}C_0^r(h)\leq \frac{2\delta^{2l}}{(2l+2)!}C_0^r(h),
\end{split}
\end{eqnarray*}
which  yields
\begin{eqnarray*}
\sum_{l=1}^\infty C_l^r(h)u_0^{(2l+2)}(x_i)=\mathcal{O}\left(\delta^2\right)~~{\rm as}~~ \delta,h\rightarrow 0.
\end{eqnarray*}

From the above equations, we can rewrite \eqref{4.5555} as
\begin{equation*}
\mathcal{L}^S_{\widetilde{\delta},h}u_0(x_i)=u_0''(x_i)+\mathcal{O}\left(\delta^2\right)\rightarrow C_0u_0''(x_i)~~{\rm as}~~\delta,h\rightarrow 0
\end{equation*}
with $C_\delta=C_0=1$. Using \eqref{4.3}, \eqref{4.4} and Lemma \ref{lemma4.4}, we obtain
\begin{eqnarray*}
\left\|u_{\delta,h}-u_0\right\|_\infty\leq  \left\|\left(-\mathcal{L}^S_{\widetilde{\delta},h}\right)^{-1}\right\|_\infty \left\|f+\mathcal{L}^S_{\widetilde{\delta},h}u_0\right\|_\infty
\leq \left\| \left(-\mathcal{L}^S_{\widetilde{\delta},h}\right)^{-1}\right\|_\infty \left\|\mathcal{L}^S_{\widetilde{\delta},h}u_0-u_0{''}\right\|_\infty.
\end{eqnarray*}
The proof is completed.
\end{proof}
\begin{theorem}\label{theorem4.8}
Let $u_{\delta}(x) \in C^{4}(\Omega)$ with $\delta=\mathcal{O}\left(h^\beta\right)$, $\beta\geq 0$.
Let $u_{\delta,h}$ and $u_0$ be the solution of \eqref{4.4} and \eqref{4.3}, respectively. Then it holds that
\begin{equation*}
\left\|u_{\delta,h}-u_0\right\|_\infty =\mathcal{O}\left(h^{\min\left\{2,2\beta\right\}}\right),~~\beta\geq 0.~~{\rm as}~~\delta,h\rightarrow 0.
\end{equation*}
\end{theorem}
\begin{proof}
 From Lemma \ref{lemma4.6} and Lemma \ref{lemma4.7}, the desired result is obtained.
\end{proof}

\section{Numerical Experiments}
We now report results of numerical experiments including two-dimensional case which substantiate the analysis given earlier. The numerical errors are measured by the $l_\infty$ (maximum) norm.
\subsection{Numerical results for one-dimensional}
In order to get simpler benchmark solutions, we choose the exact solution of the nonlocal diffusion problem \eqref{2.1} as $u_\delta(x)=x^{2}\left(1-x^{2}\right)$. This naturally leads to a $\delta$-dependent right-hand side $f_{\delta}(x)=12x^{2}-2+\frac{6}{5}\delta^{2}$.

\subsubsection{Symmetric positive definite     system \eqref{2.18}}
\begin{table}[h]\fontsize{9.5pt}{12pt}\selectfont
\begin{center}
\caption{Convergence results of $\left\|u_{\delta,h}-u_{\delta}\right\|_\infty$ if $\delta$ is  not set as grid point for 1D}\vspace{5pt}
\begin{tabular*}{\linewidth}{@{\extracolsep{\fill}}*{9}{ccccccccc}}                                    \hline  
     $h$               & $\delta=1/3$      & Rate          & $\delta=\sqrt{h}$    & Rate       & $\delta=10h/3$      & Rate           & $\delta=h^2$    & Rate          \\ \hline
$\frac{1}{80}$         & 1.1745e-03        &               & 4.1686e-04           &            & 5.1007e-05          &                & 2.4500e-05      &               \\
$\frac{1}{320}$        & 3.0069e-04        & 0.982         & 4.8211e-05           & 1.556      & 3.1109e-06          & 2.017          & 1.5273e-06      & 2.001         \\
$\frac{1}{1280}$       & 7.5619e-05        & 0.995         & 5.1976e-06           & 1.606      & 1.9325e-07          & 2.004          & 9.5412e-08      & 2.000         \\ \hline
\end{tabular*}\label{tab:aaa2}
\end{center}
\end{table}
\begin{table}[h]\fontsize{9.5pt}{12pt}\selectfont
\begin{center}
\caption{Convergence results of $\left\|u_{\delta,h}-u_{\delta}\right\|_\infty$ if $\delta$ is set as grid point/$\delta< h$ for 1D}\vspace{5pt}
\begin{tabular*}{\linewidth}{@{\extracolsep{\fill}}*{9}{ccccccccc}}                                    \hline  
     $h$              & $\delta=1/4$      & Rate          & $\delta=\sqrt{h}$    & Rate       & $\delta=5h$         & Rate           & $\delta=h^2$    & Rate          \\ \hline
$\frac{1}{16}$        & 1.9297e-06        &               & 1.9297e-06           &            & 1.3190e-06          &                & 6.1954e-04      &               \\
$\frac{1}{64}$        & 7.8439e-09        & 3.971         & 2.7206e-08           & 3.074      & 6.5949e-08          & 2.161          & 3.8314e-05      & 2.007         \\
$\frac{1}{256}$       & 3.0947e-11        & 3.992         & 3.9959e-10           & 3.044      & 3.8895e-09          & 2.041          & 2.3869e-06      & 2.002         \\ \hline
\end{tabular*}\label{tab:aaa1}%
\end{center}
\end{table}

Table \ref{tab:aaa2} shows that the shifted-symmetric quadric polynomial collocation method \eqref{2.18} has the convergence rate $\mathcal{O}\left(h^{\min\left\{2,1+\beta\right\}}\right)$ if $\delta$ is not set as a grid point.
However, it shall recover $\mathcal{O}\left(h^{\max\left\{2,4-2\beta\right\}}\right)$ when $\delta$ is set as a grid point (see Table \ref{tab:aaa1}), which is in agree with Theorem \ref{theorem4.5}.

\subsubsection{ Nonsymmetric indefinite    system \eqref{adas2.22}}
\begin{table}[h]\fontsize{9.5pt}{12pt}\selectfont
\begin{center}
\caption{Convergence results of $\left\|u_{\delta,h}-u_{\delta}\right\|_\infty$ if $\delta$ is  not set as grid point for 1D}\vspace{5pt}
\begin{tabular*}{\linewidth}{@{\extracolsep{\fill}}*{9}{ccccccc}}                                    \hline  
     $h$               & $\delta=1/3$      & Rate          & $\delta=\sqrt{h}$    & Rate       & $\delta=10h/3$      & Rate                  \\ \hline
$\frac{1}{80}$         & 1.1787e-03        &               & 4.1864e-04           &            & 5.1796e-05          &                       \\
$\frac{1}{320}$        & 3.0095e-04        & 0.984         & 4.8261e-05           & 1.558      & 3.1486e-06          & 2.020             \\
$\frac{1}{1280}$       & 7.5635e-05        & 0.996         & 5.1990e-06           & 1.607      & 1.9556e-07          & 2.004                \\ \hline
\end{tabular*}\label{tab:add22}
\end{center}
\end{table}

\begin{table}[h]\fontsize{9.5pt}{12pt}\selectfont
\begin{center}
\caption{Convergence results of $\left\|u_{\delta,h}-u_{\delta}\right\|_\infty$ if $\delta$ is set as grid point/$\delta< h$ for 1D}\vspace{5pt}
\begin{tabular*}{\linewidth}{@{\extracolsep{\fill}}*{9}{ccccccc}}                                    \hline  
     $h$              & $\delta=1/4$      & Rate          & $\delta=\sqrt{h}$    & Rate       & $\delta=5h$         & Rate                 \\ \hline
$\frac{1}{16}$        & 8.2200e-06        &               & 8.2200e-06           &            & 1.3190e-06          &                        \\
$\frac{1}{64}$        & 3.2814e-08        & 3.984         & 1.1071e-07           & 3.107      & 6.5949e-08          & 2.213             \\
$\frac{1}{256}$       & 1.2884e-10        & 3.996         & 1.6055e-09           & 3.053      & 3.8895e-09          & 2.048                \\ \hline
\end{tabular*}\label{tab:add21}%
\end{center}
\end{table}
Tables \ref{tab:add22} shows that the standard  quadric polynomial collocation method \eqref{adas2.22} has the convergence rate $\mathcal{O}\left(h^{\min\left\{2,1+\beta\right\}}\right)$ if $\delta$ is not set as a grid point. And it shall recover $\mathcal{O}\left(h^{\max\left\{2,4-2\beta\right\}}\right)$ when $\delta$ is set as a grid point  in Table \ref{tab:add21}.
However, the discrete maximum principle is not satisfied for nonsymmetric indefinite    system \eqref{adas2.22}, which might be trickier for the stability analysis of the high-order numerical schemes.
\subsubsection{Asymptotic compatibility system \eqref{4.4}}
We now choose the exact solution of the local diffusion problem \eqref{4.4} as $u_0(x)=x^{2}\left(1-x^{2}\right)$ and we can find the right-hand side $f_{0}=12x^{2}-2$.
\begin{table}[h]\fontsize{9.5pt}{12pt}\selectfont
\begin{center}
\caption{Convergence results of $\left\|u_{\delta,h}-u_{0}\right\|_\infty$ with asymptotic compatibility for 1D.}\vspace{5pt}
\begin{tabular*}{\linewidth}{@{\extracolsep{\fill}}*{9}{ccccccccc}}                                    \hline  
     $h$              & $\delta=\sqrt{h}$ & Rate          & $\delta=\sqrt[4]{h}$      & Rate       & $\delta=10h/3$      & Rate           & $\delta=h^2$    & Rate          \\ \hline
$\frac{1}{80}$        & 1.6675e-03        &               & 2.2613e-02       &            & 2.1804e-04          &                & 2.4504e-05      &               \\
$\frac{1}{320}$       & 4.4945e-04        & 0.945         & 1.0763e-02       & 0.535      & 1.3298e-05          & 2.017          & 1.5273e-06      & 2.002         \\
$\frac{1}{1280}$      & 1.1576e-04        & 0.978         & 5.0377e-03       & 0.547      & 8.2608e-07          & 2.004          & 9.5413e-08      & 2.000         \\ \hline
\end{tabular*}\label{tab:aaa3}
\end{center}
\end{table}

Table \ref{tab:aaa3} shows that the shifted-symmetric quadric polynomial collocation method \eqref{4.4} has the convergence rate  $\mathcal{O}\left(h^{\min\left\{2,2\beta\right\}}\right)$, which is in agree with Theorem \ref{theorem4.8}.

\subsubsection{Symmetric positive definite     system \eqref{2.18} with weak regularity}
Consider the nonlocal diffusion problem \eqref{2.1} with a $\delta$-dependent right-hand side $f_{\delta}(x)=e^{\delta x \sin(\pi x)}$ and homogeneous boundary conditions $g_{\delta}(x)=0$.
Since the analytic solutions is unknown, the order of the convergence of the numerical results are computed by the following formula
\begin{equation*}
  {\rm Convergence ~Rate}=\frac{\ln \left(||u_{\delta,4h}-u_{\delta,h}||_\infty/||u_{\delta,h}-u_{\delta,h/4}||_\infty\right)}{\ln 4}.
\end{equation*}

\begin{table}[h]\fontsize{9.5pt}{12pt}\selectfont
\begin{center}
\caption{Convergence results of $\left\|u_{\delta,h}-u_{\delta}\right\|_\infty$ if $\delta$ is  not set as grid point for 1D}\vspace{5pt}
\begin{tabular*}{\linewidth}{@{\extracolsep{\fill}}*{9}{ccccccc}}                                    \hline  
     $h$               & $\delta=1/3$      & Rate          & $\delta=\sqrt{h}$    & Rate       & $\delta=10h/3$      & Rate                     \\ \hline
$\frac{1}{20}$         & 3.0314e-02        &               & 3.9558e-02           &            & 4.1090e-02          &                         \\
$\frac{1}{80}$         & 7.6372e-03        & 0.994         & 1.4563e-02           & 0.720      & 1.0150e-02          & 1.008                   \\
$\frac{1}{320}$        & 3.2475e-03        & 0.616         & 6.5622e-03           & 0.575      & 2.5371e-03          & 1.000              \\ \hline 
\end{tabular*}\label{tab:add2}
\end{center}
\end{table}

Table \ref{tab:add2} shows that the convergence rate of the shifted-symmetric quadric polynomial collocation method \eqref{2.18} decreases when $\delta$ is not a grid point.
In fact, it is not possible to reach high-order convergence even with the high order scheme   because of the weak regularity of the solution in the region close to  the boundaries.
In order to restore the desired convergence rate with the nonsmooth  data or weak regularity of the solution , it may need designed the corrected algorithms as  nonlocal fractional problems \cite {ShCh:2020}.

%
%
%

\subsection{Numerical results for two-dimensional}
In order to get simpler benchmark solutions, we choose the exact solution of the nonlocal diffusion problem (\ref{2.1}) as $u_\delta(x,y)=x^2(1-x^2)+y^2(1-y^2)$ with $\gamma_{{\delta}}(z)=\frac{3}{2{\delta}^{4}}$. This naturally leads to a $\delta$-dependent right-hand side $f_{\delta}(x,y)$ and $f_0(x,y)$ in \cite{LTTF:21}.


\begin{table}[h]\fontsize{9.5pt}{12pt}\selectfont
\begin{center}
\caption{Convergence results of $\|u_{\delta,h}-u_{\delta}\|_\infty$ with  general  $\delta$ for 2D}\vspace{5pt}
\begin{tabular*}{\linewidth}{@{\extracolsep{\fill}}*{9}{ccccccccc}}                                    \hline  
     $h$              & $\delta=1/4$      &Rate         & $\delta=\sqrt{h}$    &Rate         & $\delta=h^2$  &Rate   & $\delta=1/3$      &Rate     \\\hline
$\frac{1}{4}$         & 5.0933e-04        &             & 1.7804e-04           &                 &1.1969e-02     &      & 8.9981e-03        &     \\
$\frac{1}{16}$        & 2.3224e-06        & 3.888       & 2.3224e-06           &3.130        &7.3135e-04      &2.016     & 3.2925e-03        & 0.725    \\
$\frac{1}{64}$        & 9.4308e-09        & 3.972       & 3.2255e-08           &3.085        &4.5166e-05      &2.008      & 8.9298e-04        & 0.941  \\\hline 
\end{tabular*}\label{tab:aaa2d1}%
\end{center}
\end{table}
Table \ref{tab:aaa2d1} shows that the shifted-symmetric quadric polynomial collocation method \eqref{bd2.24} has the convergence rate
$\mathcal{O}\left(h^{\min\left\{2,1+\beta\right\}}\right)$ if $\delta$ is not set as  the grid point and $\mathcal{O}\left(h^{\max\left\{2,4-2\beta\right\}}\right)$
if $\delta$ is set as the grid point.

\section{Conclusions}
The stability of the piecewise quadratic polynomial (higher-order) collocation methods on nonlocal model  is not a trivial task, since the algebraic system is  nonsymmetric and indefinite \cite{CCNWL:20,Chen:1--30} and lack of a discrete maximum principle \cite{DDGGTZ:20,LTTF:21}.
In the first part of this work, we construct the shifted-symmetric piecewise quadratic polynomial collocation method for nonlocal model, which has the symmetric positive definite system and satisfies  the discrete maximum principle. Then a numerical solution of the presented method approximation  to both the continuum nonlocal solution  and its local limit has been proven strictly.
In the second part of this work, we design the numerical schemes when $\delta$ is not set as the grid point and independent of the computational mesh size $h$ \cite{WT:7730--7738}.
The detailed proof of the convergence analysis for the nonlocal models with the general  horizon parameter  are provided.
We remark that the error estimate is carried our by Taylor series with a strong assumption on the regularity of exact solution \cite{Du:2017}. It is of great interest to improve the error estimate
under the low regularity assumption. There is also interesting to design the numerical algorithms which restore or keep high-order schemes  if  $\delta$ is not set as the grid point.

\noindent {\bf Acknowledgments.} This work was supported by NSFC 11601206 and NSFC 11671165. The  authors wish to  thank  Prof. Qiang Du and Xiaochuan Tian for them valuable comments
and Dr. Rongjun Cao for simulating the multidimensional model.

%
%
%
%
%
{
\small
\bibliographystyle{ieee}

}

\end{document}